\documentclass[12pt]{amsart}

\usepackage{amsmath,amssymb,amsthm}

\usepackage{hyperref}

\numberwithin{equation}{section}
\theoremstyle{plain}
\newtheorem{theorem}[equation]{Theorem}

\newtheorem{proposition}[equation]{Proposition}
\newtheorem{lemma}[equation]{Lemma}
\newtheorem{corollary}[equation]{Corollary}

\theoremstyle{remark}

\newtheorem{remark}[equation]{Remark}

\theoremstyle{definition}
\newtheorem{definition}[equation]{Definition}

\newcommand{\lra}{\longrightarrow}
\newcommand{\ra}{\to}
\newcommand{\restr}{\mbox{\Large \(|\)\normalsize}}

\newcommand{\C}{{\mathcal C}}
\renewcommand{\c}{\mathbb C}
\newcommand{\D}{{\mathcal D}}

\newcommand{\F}{{\mathcal F}}
\newcommand{\g}{{\mathcal G}}

\renewcommand{\H}{\mathbb H}

\renewcommand{\l}{{\mathcal L}}
\renewcommand{\L}{\mathbb L}

\newcommand{\N}{\mathbb N}
\renewcommand{\P}{{\mathbb P}}

\newcommand{\R}{\mathbb R}

\renewcommand{\S}{{\mathcal S}}

\newcommand{\Z}{\mathbb Z}

\newcommand{\cent}{\operatorname{Center}}

\newcommand{\cut}{\operatorname{Cut}}
\newcommand{\dmu}{\D_\mu}

\newcommand{\excess}{\operatorname{excess}}

\newcommand{\horiz}{\operatorname{hor}}
\newcommand{\hs}{\operatorname{HS}}

\newcommand{\Int}{\operatorname{Int}}
\newcommand{\Intmu}{\operatorname{Int_\mu}}
\newcommand{\isom}{\operatorname{Isom}}

\newcommand{\Lip}{\operatorname{Lip}}

\newcommand{\loc}{\operatorname{loc}}

\newcommand{\realpart}{\operatorname{Re}}

\newcommand{\spt}{\operatorname{spt}}

\newcommand{\width}{\operatorname{width}}

\def\D{\partial}
\newcommand{\al}{\alpha}
\def\de{\delta}
\def\De{\Delta}

\def\eps{\epsilon}
\def\ga{\gamma}
\def\Ga{\Gamma}

\def\la{\lambda}

\def\lra{\longrightarrow}

\def\om{\omega}

\def\ra{\to}

\def\si{\sigma}
\def\Si{\Sigma}

\def\be{\beta}

\def\th{\theta}

\def\defeq{:=}

\newcommand{\ol}{\overline}

\def\XXint#1#2#3{{\setbox0=\hbox{$#1{#2#3}{\int}$}
      \vcenter{\hbox{$#2#3$}}\kern-.5\wd0}}

\begin{document}

\title[Metric differentiation, monotonicity, and $L^1$]{Metric differentiation,
monotonicity and maps to~$L^1$}

\begin{abstract}
This is one of a series of papers on Lipschitz maps from metric spaces
to $L^1$.  Here we present the details of results which were 
announced in  \cite[Section 1.8]{ckbv}: a new approach
to the infinitesimal structure of Lipschitz maps
into $L^1$, and, as a first application,
an alternative proof of the main theorem of \cite{ckbv},
that the Heisenberg group does not
admit a bi-Lipschitz embedding in $L^1$.  The proof 
uses the metric differentiation theorem of Pauls \cite{pauls}
and the cut metric description in
\cite{ckbv} to reduce the nonembedding argument
to a classification of monotone subsets of the
Heisenberg group.  A quantitative version of this
classification argument is used
in our forthcoming joint paper with Assaf Naor \cite{ckn}.
\end{abstract}

\date{\today}
\author{Jeff Cheeger}
\thanks{Research supported in part by NSF grant DMS-0704404}
\author{Bruce Kleiner}
\thanks{Research supported in part by NSF grant DMS-0805939}
\maketitle

\setlength{\parskip}{.5ex}

{\small\tableofcontents}

\setlength{\parskip}{1.5ex}

\section{Introduction}
\label{sec-intro}

In this paper, we continue our  investigation of Lipschitz maps of metric
spaces into Banach spaces \cite{ckbv,ckdppi,laaksoembed}, which is 
motivated  by the role of bi-Lipschitz embedding problems in theoretical
computer science \cite{liniallondonrabinovich,aumannrabani,linialicm,leenaor},  
earlier developments in the infinitesimal geometry of 
metric measure spaces \cite{pansu,Cheeger}, and the geometry of Banach
spaces \cite[Chapters 6-7]{benlin},\cite{bourgain}.  Our main purpose here is to 
present the details of an approach to
Lipschitz maps into $L^1$ announced in \cite[Section 1.8]{ckbv},
which gives new insight into both 
embeddability and non-embeddability questions;  as a first application,
we give a new proof of
a (slightly stronger version of) the main result of \cite{ckbv}.  Other
implications will be pursued in subsequent papers, see below for more discussion.

\bigskip\bigskip
\subsection*{An approach to the infinitesimal structure of Lipschitz maps into
$L^1$}
We begin by describing the  approach in general terms, cf 
\cite[Section 1.8]{ckbv}.  

Let $X$
be a metric  space equipped with a Borel measure $\mu$, and suppose $f:X\ra L^1$
is a Lipschitz map.
Under certain assumptions on the pair $(X,\mu)$ (e.g. if it is doubling and
satisfies a Poincare
inequality \cite{HeKo}), one can prove a generalization of
Kirchheim's metric differentiation theorem
\cite{kirchheim}.  This says, roughly speaking, that
for $\mu$-a.e. $x\in X$,  if one blows up $f$ near $x$ then it looks more and
more like a geodesic map, when restricted to certain curves.  Passing to 
limits, one arrives at a new map $f_\infty:X_\infty\ra L^1$, where $X_\infty$
is blow-up of $X$ containing a distinguished class of geodesics called
 {\bf lines}, and the restriction of $f_\infty$ to each line $L\subset X_\infty$
gives a constant speed parametrization of some  geodesic $f_\infty(L)\subset L^1$,
i.e. for all $x_1,\,x_2\in L$,
\begin{equation}
\label{eqn-geodesic}
\|f_\infty(x_1)-f_\infty(x_2)\|_{L^1}=c_L\,d(x_1,x_2)\,,
\end{equation}
 where $c_L\in [0,\infty)$ is a constant depending
on $L$.
Here $f_\infty$ has Lipschitz constant (respectively bi-Lipschitz constant) no 
larger than that of $f$.  

Let $\rho_\infty:X_\infty\times X_\infty\ra \R$ denote the
pseudo-distance given by $\rho_\infty(x_1,x_2)=\|f_\infty(x_1)-f_\infty(x_2)\|_{L^1}$.
Appealing  to  \cite{ckbv,assouad}, one obtains a representation of 
$\rho_\infty$ 
as a superposition of elementary cut metrics \cite[Section 3]{ckbv}:
\begin{equation}
\label{eqn-introcutmetric}
\rho_\infty=\int_{\cut(\H)}\;d_E\,d\Si(E),\,
\end{equation}
where $\cut(X_\infty)$ is the collection of (equivalence classes of) measurable
subsets of $X_\infty$,  $d_E:X_\infty\times X_\infty\ra \{0,1\}$ is defined by 
$d_E(x_1,x_2)=|\chi_E(x_1)-\chi_E(x_2)|$, and $\Si$ is a 
measure on $\cut(X_\infty)$.   

The geodesic property (\ref{eqn-geodesic}) turns out to be 
 equivalent  to the condition that $\Si$-a.e.  $E\in \cut(X_\infty)$
is {\bf monotone}, which means that   for almost every line
$L\subset X_\infty$, the characteristic function $\chi_E$ restricted
to $L$ agrees almost everywhere (with respect to  linear measure
on $L$)
with a monotone function (Proposition \ref{prop-geodesicmonotone}).  
Thus questions about bi-Lipschitz embedding lead directly to an
investigation of monotone sets in blow-up spaces, and some instances
this leads to a complete resolution.  In this paper we implement this
approach when $X$ is the Heisenberg group, and in \cite{laaksoembed}
we use it to exhibit embeddings of Laakso-type spaces into $L^1$.

We note that Lee-Raghavendra \cite{leeraghavendra}
have used a similar combination of ideas in the context of
finite  graphs:  they use a form of metric differentiation -- the 
coarse differentiation of Eskin-Fisher-Whyte \cite{eskinfisherwhyte} --
together with
essentially the same notion of monotonicity as above.  Using this
argument, they show that a certain 
family of series-parallel graphs has supremal
$L^1$ distortion equal to $2$,  which matches the 
known upper bound on distortion for this family of 
graphs \cite{chakrabarti}.

\bigskip\bigskip
\subsection*{Lipschitz maps from the Heisenberg group}
Let $\H$ denote the Heisenberg group equipped  with the
Carnot-Caratheodory metric $d$.
It was shown in \cite{ckbv} that  metric balls
$B\subset\H$ 
do not
bi-Lipschitz embed in $L^1$.
More specifically, it was shown that if $f:B\to L^1$ is any Lipschitz map, 
then blowing $f$ up at a generic point $x\in B$,
one obtains a family of maps which degenerate
along cosets of the center of $\H$, which implies that $f$ is 
not bi-Lipschitz. 
In this paper we give a  shorter and
largely self-contained proof of the nonembedding result, as well
as a strengthening of the main result of \cite{ckbv}, using the
approach indicated above.

Let $f:\H\ra L^1$ be a Lipschitz map, and let $\rho:\H\times\H\ra [0,\infty)$
be defined by
$\rho(x_1,x_2)=\|f(x_1)-f(x_2)\|_{L^1}$, i.e. $\rho$ is the 
pullback of the distance on $L^1$ by $f$.  If $x\in \H$
and $\la\in (0,\infty)$, let $\rho_{x,\la}$ be the result
of dilating $\rho$ at $x$, and renormalizing:
$$
\rho_{x,\la}(z_1,z_2)
\,=\,\frac{1}{\la}\,\rho( x\,s_{\la}z_1,x\,s_{\la}z_2)
\,=\, \frac{1}{\la} ((s_\la)^*(\ell_x)^*\rho)(z_1,z_2)
\,,
$$
where $s_\la:\H\ra\H$
is the automorphism which scales distances by the factor $\la$,
and $\ell_x:\H\ra\H$ is left translation by $x$.

\begin{theorem}
\label{thm-maindiff}
For almost every
$x\in \H$, there is a semi-norm $\|\cdot\|_x$ on $\R^2$ such that
$$
\rho_{x,\la}(z_1,z_2)\ra \|\pi(z_1)-\pi(z_2)\|_x\quad
\mbox{as}\quad\la\ra 0,\,
$$
 uniformly on compact subsets of $\H\times\H$.  Here
$\pi:\H\ra \H/[\H,\H]\simeq \R^2$ is the abelianization homomorphism. 
In particular, $\rho_{x,\la}$ converges to a pseudo-distance which 
is zero along fibers of $\pi$, and hence $f$ is not bi-Lipschitz in any
neighborhood of $x$.
\end{theorem}

\bigskip\bigskip
\subsection*{Discussion of the proof}
We use the term {\bf line} to refer to a coset 
$g\exp\R X$ of a horizontal $1$-parameter subgroup
$\exp \R X\subset \H$, and we refer to a pair of 
points $(z_1,z_2)\in\H\times\H$ as {\bf horizontal} if it
lies on a line, see Section \ref{sec-prelim}.

The first step in the  proof of Theorem \ref{thm-maindiff}
is to invoke the metric differentiation
theorem of Pauls \cite{pauls} (see Theorem \ref{thm-pauls}).
This  guarantees that for almost every  $x\in \H$, there
is a semi-norm $\|\cdot\|_x$ on $\R^2$ such that the statement of
Theorem \ref{thm-maindiff} holds provided we restrict to horizontal pairs
$(z_1,z_2)$, 
i.e.
\begin{equation}
\label{eqn-horizontalgood}
\rho_{x,\la}(z_1,z_2)\ra \|\pi(z_1)-\pi(z_2)\|_x\quad
\mbox{as}\quad \la\ra 0\,,
\end{equation}
with uniform convergence on compact sets of horizontal pairs.
  The remainder of the argument is 
devoted to showing that (\ref{eqn-horizontalgood}) holds for all pairs
$(z_1,z_2)\in\H\times\H$,
not just horizontal pairs.  If   this were false, then using the 
fact that $\rho_{x,\la}\leq \Lip(f)\,d$ for all $(x,\,\la)\in\H\times (0,\infty)$, we
may apply the Arzela-Ascoli theorem to find a sequence 
$\{\la_k\}\ra 0$ such that the sequence of pseudo-metrics
$\{\rho_{x,\la_k}\}$ converges uniformly on compact subsets of $\H\times \H$ 
to a pseudo-distance $\rho_\infty$, where:
\begin{enumerate}
\item
$\rho_\infty(z_1,z_2)\,=\,\|\pi(z_1)-\pi(z_2)\|_x$ for all horizontal pairs
$(z_1,z_2)$.
\item   $\rho_\infty(\bar z_1,\bar z_2)\neq
\|\pi(\bar z_1)-\pi(\bar z_2)\|_x$ for some $(\bar z_1,\bar z_2)\in\H\times\H$.
\item $\rho_{\infty}\,\leq\,\Lip(f)\,d\,$.
\end{enumerate}

Next, we apply ultralimits (or ultraproducts in the Banach space literature)
and a theorem of Kakutani \cite{kakutani},
 to see that
$\rho_\infty$ is also induced by a Lipschitz map $f_\infty:\H\ra L^1$.
Therefore
$\rho_\infty$ has a cut metric representation 
as a superposition of elementary cut metrics (\ref{eqn-introcutmetric}).

Condition (1) implies that the restriction 
of $f_\infty$ to any line
 gives a constant speed parametrization of some geodesic in $L^1$, and as
mentioned above,
 this 
property of $f_\infty$
 is equivalent  to the condition that $\Si$-a.e.  $E\in \cut(\H)$
is {\bf monotone}:   for almost every line
$L\subset\H$, the characteristic function $\chi_E$ restricted
to $L$ agrees almost everywhere (with respect to  linear measure
on $L$)
with a monotone function (Proposition \ref{prop-geodesicmonotone}).  

Most of the work in  the proof goes into Theorem
\ref{thm-monotonehalfspace}, which classifies monotone subsets
of $\H$.  The monotone subsets of $\H$ turn out to be
the half-spaces, modulo sets
of measure zero.  A {\bf half-space} in $\H$ is a connected component
of  $\H\setminus P$ where $P$ is either a vertical plane (a coset
of a subgroup isomorphic to $\R^2$), or a horizontal
plane (the union of the lines passing through some point $g\in \H$).

Thus the cut measure $\Si$ is supported on half-spaces.
We then show that  $\Si$ is in fact supported 
on vertical half-spaces (Section \ref{sec-injectivity}). This involves 
proving the injectivity
of a certain convolution operator on $\H$, and invokes some
harmonic analysis  results from  \cite{strichartz}.
Finally, for cut measures supported on vertical half-spaces, the cut
metric $\rho_\infty(z_1,z_2)$
depends only on the projections $\pi(z_1),\,\pi(z_2)$, 
which contradicts (2).

We would like to emphasize 
 that there is a simpler way to conclude the argument which avoids
the harmonic analysis in Section \ref{sec-injectivity}, if
one is only interested in the bi-Lipschitz nonembedding result.  We present this alternate
endgame in Section
\ref{sec-proofweak}.

\bigskip
\subsection*{Comparison  with \cite{ckbv}}
Both the proof given here and the original proof in \cite{ckbv} use the 
cut metric representation, 
together with a differentiation argument.
Here we use the  $L^1$ cut metric representation
\cite[Section 3]{ckbv}, rather than
the finer representation using sets of finite perimeter in \cite[Section 4]{ckbv}.
Also, we use the  differentiation result of \cite{pauls}, rather than the
differentiation results  \cite{luigi1,luigi2,italians1,italians2}, which were
a key ingredient in \cite{ckbv}.  This leads to a much stronger
restriction on the cuts showing up in the cut representation of
the blown-up map as compared with the original map -- they
are monotone, rather than arbitrary sets of locally finite perimeter.
We point out that the classification proof for monotone sets has some similarities
with the classification proof for  sets with constant normal, an important
component of  \cite{italians1,italians2}.   The harmonic analysis
material appearing in Section \ref{sec-injectivity}
does not seem to correspond to anything in \cite{ckbv}.

Apart from providing a substantially different approach from
\cite{ckbv} to (generalized) differentiability
theory for Lipschitz maps into $L^1$, 
the argument here gives a stronger conclusion, and is significantly shorter
than \cite{ckbv}.
The bi-Lipschitz nonembedding proof via the  Theorem \ref{thm-weakdiff}
is much shorter than \cite{ckbv}, and is self-contained, apart from foundational
material on $L^1$-cut measures taken from \cite[Section 3]{ckbv}.

\bigskip
\subsection*{Further results}
In a forthcoming paper with Assaf Naor \cite{ckn}, we prove a quantitative version of
the nonembedding theorem,
i.e.  for every $\eps>0$ we find
 an explicit  $\delta=\delta(\epsilon)$, such that for
every $1$-Lipschitz map $f:B\to L^1$,
there exist $x_1,x_2$ with $d(x_1,x_2) > \delta$ and
$\|f(x_1)-f(x_2)\|_{L_1}<\epsilon \cdot d(x_1,x_2)$.
Central to the argument is the formulation and proof of a quantitative
version of the classification of monotone sets given in this paper.
It is also necessary to estimate,  in terms of $\epsilon$, a scale on which
this quantitative classification can be applied;  sets of finite
perimeter a play a
direct role in this step of the argument.

The rough outline of the first part of the 
proof given here is applicable in much greater
generality, in particular to a large family of spaces satisfying Poincare
inequalities.   
 In this broader context there is a version of metric differentiation
\cite{ckmetdiff},
as well as an associated notion of monotone sets, which can be used to
study bi-Lipschitz embedding in $L^1$; the final conclusions about
embeddability or nonembedding depend on the
the structure of monotone sets, which varies from example to example.
For instance, in contrast to the Heisenberg group,  the  Laakso spaces
bi-Lipschitz embed in $L^1$, even though they do not bi-Lipschitz embed
in Banach spaces satisfying the Radon-Nikodym property, such as the space
of sequences $\ell^1$, see \cite{laaksoembed}.
We will pursue these ideas elsewhere.

\subsection*{Organization of the paper}
Section \ref{sec-prelim} collects some background material.  In Section 
\ref{sec-monotonegeodesic} we relate the geodesic property of maps 
$X\ra L^1$ with the monotonicity of the associated cut measure.
In Section \ref{sec-preciselymonotone}, we  classify precisely
monotone subsets of $\H$; this argument
has fewer technical complications, but the same outline as the proof
Theorem \ref{thm-monotonehalfspace}.  In Section \ref{sec-monotone}
we prove Theorem \ref{thm-monotonehalfspace}.  In Section 
\ref{sec-proofweak} we prove  Theorem \ref{thm-weakdiff}.
In Section \ref{sec-injectivity} we analyze the linear operator $\Si\mapsto d_{\Si}$
which assigns a cut metric to a signed cut measure.
In Section \ref{sec-standardonlines} we complete the proof of Theorem
\ref{thm-maindiff}.

\subsection*{Acknowledgements}
 We would like to thank John Mackay for a number of comments on the first
version of this paper, in particular for noticing that Definition 5.3 
was incomplete.

\section{Preliminaries}
\label{sec-prelim}

In this section we recall various facts that will be needed later,
and fix notation.

We will use $\l$ as a generic symbol 
to denote Haar measure on Lie groups and associated homogeneous spaces.

\subsection{Carnot groups}

We recall that a {\bf Carnot group} is a triple 
$(G,\De,\langle\,\cdot,\cdot\,\rangle)$,
where $G$ is a simply-connected nilpotent Lie group,
$\De$ is a subspace of the Lie algebra of $G$,
$\langle\,\cdot,\cdot\,\rangle$ is an inner product on $\De$,
and there is a decomposition of the Lie algebra of $G$
as a direct sum
$$
L(G)\;=\;V_1\oplus\ldots\oplus V_k,
$$
where $V_1=\De$, and $[V_1,V_i]=V_{i+1}$ 
for all 
$i\in \{1,\ldots,k-1\}$.   
For every $\la\in (0,\infty)$,
there is a unique automorphism $s_{\la}:G\ra G$ whose derivative
scales $V_i$ by the factor $\la^i$.
The direct sum $V_2\oplus\ldots\oplus V_k$ 
is an ideal in the
Lie algebra $L(G)$, which  is the  tangent space of 
the derived  subgroup $[G,G]\subset G$.  We denote the
canonical epimorphism to the abelianization of $G$
by $\pi:G\ra G/[G,G]$; the latter is just a copy of $\R^n$
for $n=\dim \De$.  

We will also view $\De$
as a left invariant distribution on $G$ (or left invariant 
sub-bundle of $TG$),
and refer to it as  the {\bf horizontal space}.  
A $C^1$ path $c:I\ra G$ is {\bf horizontal} if
its velocity is tangent to $\De$ everywhere.  A 
horizontal path
$c$  is a {\bf horizontal
lift} of a path $\bar c:I\ra G/[G,G]$ if $\bar c=\pi\circ c$.
Given a $C^1$ path $\bar c:I\ra G/[G,G]$, 
$t\in I$, and $x\in \pi^{-1}(\bar c(t))$, there is a unique
horizontal lift $c$ of $\bar c$ such that $c(t)=x$.
A {\bf  line} is the image of a horizontal lift of a straight
line in $G/[G,G]\simeq \R^n$, or equivalently, a 
line is a subset $L\subset G$ of the form
$$
L=\{g\,\exp(tX)\mid t\in \R\}
$$
for some $g\in G$, $X\in \De\setminus\{0\}$, or to put is another
way,  a line is
a left translate of a (nontrivial) horizontal $1$-parameter 
subgroup.  We let $\L(G)$ denote the collection of 
all lines in $G$; this has a natural smooth 
structure.  
A {\bf horizontal pair} is pair of points $x_1,\,x_2\in X$
which lie on a line.  We let $\horiz(G)\subset G\times G$ denote the collection
of horizontal pairs; this is a closed subset of $G\times G$.

We equip $G$ with
the  Carnot-Carath\'eodory (or sub-Riemannian) distance function
$d_G$ associated with the pair 
$(\De,\langle\,\cdot,\cdot\,\rangle)$, namely $d_G(p,q)$ is the infimal
length of a horizontal path joining $p$ to $q$.

\bigskip\bigskip
\subsection{The Heisenberg group}
\label{sec-heisenberg}
Recall that the $3$-dimensional Heisenberg group
$\H$ is the matrix group
$$
\left\{\left.\begin{bmatrix}
  1 & a&c\\
 0&1&b\\
 0&0&1
\end{bmatrix}\;\right| \;a,b,c\in\R\right\}\,,
$$
whose Lie algebra of $\H$ has the presentation
$$
[X,Y]=Z,\quad [X,Z]=[Y,Z]=0\,,
$$
where

$$
X=\begin{bmatrix}
  0 & 1&0\\
 0&0&0\\
 0&0&0
\end{bmatrix}\,,
\quad 
Y=\begin{bmatrix}
  0 & 0&0\\
 0&0&1\\
 0&0&0
\end{bmatrix}\,
\quad
Z=\begin{bmatrix}
  0 & 0&1\\
 0&0&0\\
 0&0&0
\end{bmatrix}\,.
$$

\noindent
We will identify $X,Y$, and $Z$ with left invariant
vector fields on $\H$.
The Carnot group structure on $\H$ is the triple
$(\H,\De,\langle\,\cdot,\cdot\,\rangle_{\De})$, where $\De$ is the
$2$-dimensional subbundle of the tangent bundle $T\H$
spanned by $\{X,Y\}$, and 
$\langle\,\cdot,\cdot\,\rangle_{\De}$ is the left invariant
Riemannian metric on $\De$ for which $\{X,Y\}$ are orthonormal.
The center of $\H$ is the $1$-parameter group 
$\{\exp tZ\mid t\in \R\}$, which is also the derived subgroup
$[\H,\H]$.  The canonical
epimorphism to the abelianization $\pi:\H\ra \H/[\H,\H]=\H/\cent(\H)$ 
will be identified with the homomorphism
$\pi:\H\ra\R^2$ where

$$
\pi\left(\begin{bmatrix}
  1 & a&c\\
 0&1&b\\
 0&0&1
\end{bmatrix}\right)=(a,b)\,.
$$

\bigskip
\begin{lemma}
\label{lem-enclosedarea}
If $c:I\ra \H$ is a horizontal lift of a loop $\bar c:I\ra \R^2$,
then the endpoints $c(0),c(1)$ both lie in the same
fiber $F=\pi^{-1}(\bar c(0))=\pi^{-1}(\bar c(1))$, and  satisfy
\begin{equation}
\label{eqn-heightequalsenclosedarea}
c(1)=c(0)\exp(A\,Z)\,,
\end{equation}
where $A\in \R$ is the signed Euclidean area enclosed by 
the loop $\bar c$.  
\end{lemma}
\begin{proof} 
For readers familiar with connections on principal bundles, the
projection $\pi:\H\ra\R^2$  defines a principal $\R$-bundle, and horizontal
distribution is a connection with curvature form $dx\wedge dy$.  The lemma
follows from the relation between holonomy and curvature, for abelian
principal bundles.

Here is an elementary proof.  Let  $\{\al_X,\,\al_Y,\al_Z\}$ be
the basis of of left-invariant $1$-forms 
dual to $\{X,Y,Z\}$, so 
 $\al_X=\pi^*dx$, $\al_Y=\pi^*dy$, and 
$d\al_Z=-\al_X\wedge\al_Y=-\pi^*(dx\wedge dy)$.
Let $\eta:I\ra \H$
be a path in the fiber $F$ running from $c(1)$
to $c(0)$.     If $\ga$ is the concatenation of 
$c$ and $\eta$, then $\ga$ is a cycle which bounds
a $2$-chain $\zeta$, so by Stokes' theorem we have
$$
\int_{\eta}\;\al_Z\;=\;\int_{\ga}\;\al_Z 
$$
$$
=\;\int_{\zeta}\;d\al_Z=-\int_{\pi\circ\zeta}\;dx\wedge dy=-A\,,
$$
which means that $c(0)=c(1)\exp(-AZ)$, or $c(1)=c(0)\exp\,AZ$.

\end{proof}

\bigskip
If $x,y\in \H$, then by Lemma \ref{lem-enclosedarea}
any two horizontal paths $c_1,\,c_2$ from $x$ to $y$
have projections  $\pi\circ c_1$, $\pi\circ c_2$ which
enclose zero signed area; conversely, any path $c:I\ra \H$
which starts at $x$, and ends in $\pi^{-1}(y)$ will terminate
at $y$ provided the signed area enclosed by $c_1$ and $c$ is zero.

The geodesics (locally length minimizing paths)
in $\H$ are horizontal lifts of circles in $\R^2$.  This implies
that for every $p\in \H$, 
\begin{equation}
\label{eqn-verticaldistance}
d(p,p\,\exp tZ)=\sqrt{4\pi|t|}\,,
\end{equation}
by the solution to the isoperimetric problem.

A {\bf vertical plane} is a subset $P\subset\H$ of the form 
$\pi^{-1}(L)$, where $L$ is a line in the plane.
The {\bf horizontal plane centered at   $p\in\H$} is the 
union of the lines passing through $p$.  A {\bf plane} is
a vertical or horitzontal plane. 
A {\bf vertical (respectively horizontal) half-space} is one
of the two components of $\H\setminus P$, for some vertical (respectively
horizontal) plane.

Two lines $L_1,\,L_2\in\L(\H)$ are {\bf parallel} if they are tangent to the
same horizontal vector field $X\in \De\setminus\{0\}$, or equivalently, if
$\pi(L_1),\,\pi(L_2)$ are parallel lines in $\R^2$.  Two lines are {\bf skew}
if they are disjoint and not parallel.

\begin{lemma}
\label{lem-geometry}

\mbox{}
\begin{enumerate}
\renewcommand{\theenumi}{\Alph{enumi}}
\renewcommand{\labelenumi}{(\theenumi)}
\setlength{\itemsep}{.75ex}%
\setlength{\parskip}{.75ex}%

\item 
\label{item-parallel}
Suppose $L_1,L_2\in\L(\H)$ are parallel but $\pi(L_1)\neq \pi(L_2)$.  
Then there
is a unique fiber $\pi^{-1}(x)\subset H$ lying halfway between 
$\pi(L_1)$ and $\pi(L_2)$
such that every point in $L_1$ can
be joined to $L_2$ by a unique line, and this line will pass 
through $\pi^{-1}(x)$.  Moreover every point in $\pi^{-1}(x)$
lies on a unique such line.

\item
\label{item-skew}
   Suppose $L_1,L_2\in \L(\H)$ are skew lines.
Then there is a hyperbola $Y\subset \R^2$ with asymptotes $\pi(L_1)$ 
and $\pi(L_2)$,
such that every tangent line of $Y$ has a unique horizontal lift 
which intersects
both $L_1$ and $L_2$, and conversely, if $L\in \L(\H)$ and 
$L\cap L_i\neq\emptyset$
for $i=1,2$, then $\pi(L)$ is tangent to $Y$.
\end{enumerate}
\end{lemma}
\proof
(\ref{item-parallel}).  Let $\bar\eta:[0,1]\ra\R^2$ be a line segment running
from $\pi(L_1)$ to $\pi(L_2)$, and let $\eta:[0,1]\ra\H$ be a horizontal lift of $\bar\eta$
ending in $L_2$.  Then $x_1=\eta(0)\,\exp\,AZ$ for some $A\in \R$,
where $x_1\in L_1$.  Now form a closed quadrilateral $\bar\eta,\,\bar\al,\,\bar\beta,\,
\bar\ga$ enclosing signed area $A$, where $\bar\al\subset \pi(L_2)$ and
$\bar \ga\subset \pi(L_1)$.  Horizontally lifting this to an open 
quadrilateral $\eta,\,\al,\,\beta,\,\ga$, by Lemma \ref{lem-enclosedarea}
we have $\ga\subset L_1$, which implies that $\beta$ intersects both $L_1$
and $L_2$.

Let $\bar \beta_1:[0,1]\ra\R^2$ be a line segment from $\pi(L_1)$ to $\pi(L_2)$
which passes through the midpoint $x$ of $\bar\beta$.   Then we obtain a 
closed (self-intersecting) quadrilateral $\bar\beta,\,\bar\de_1,\,\bar\beta_1,\,\bar\de_2$
enclosing zero signed area, where $\bar\de_i\subset \pi(L_i)$.
Horizontally lifting this to $\H$, by Lemma \ref{lem-enclosedarea} we get
a closed quadrilateral $\beta,\,\de_1,\,\beta_1,\,\de_2$, and hence $\beta_1$
intersects both $L_1$ and $L_2$. 

Given $A\in \R$, we may choose $\bar\beta_1$ so that the area of the triangle
enclosed by $\bar\beta_1,\,\bar\beta$, and $\pi(L_1)$ is $A$; it follows that
$\beta_1$ may be chosen to pass through any prescribed point in the fiber
$\pi^{-1}(x)$.

(\ref{item-skew}).  Let $ x=\pi(L_1)\cap \pi(L_2)$, and $x_i=L_i\cap \pi^{-1}(x)$. 
Then $x_2=x_1\exp\,AZ$ for some $A\in\R\setminus\{0\}$.  Let $\Ga$ be the
collection of line segments $\bar\eta:I\ra\R^2$ running from $\pi(L_1)$ to
$\pi(L_2)$, such that the oriented triangle with vertices $x,\,\bar\eta(0),\,
\bar\eta(1)$ encloses signed area $A$.  By Lemma \ref{lem-enclosedarea},
if $\bar\eta\in\Ga$, then the horizontal lift $\eta:I\ra\H$ starting
in $L_1$ ends on $L_2$.   The elements of $\Ga$ are precisely the segments
tangent to a hyperbola with asymptotes $\pi(L_1)$ and $\pi(L_2)$.  To see this,
apply an area-preserving affine transformation so that $\pi(L_1)$
and $\pi(L_2)$ are the $x$ and $y$ axes, respectively; then by analytic geometry,
the tangent lines to the hyperbola defined by $xy=C$ enclose area
$2C$.

\qed

\bigskip\bigskip
\subsection{Metric differentiation and blow-ups of Lipschitz maps} 
\label{subsec-metricdifferentiation}

Let $f:G\ra Y$ be a Lipschitz map from a Carnot group to a metric space,
and let $\rho:G\times G\ra [0,\infty)$ be the pullback of the distance
function, i.e. $\rho(g_1,g_2)=f^*d_Y(g_1,g_2)=d_Y(f(g_1),f(g_2))$.
We will need the following metric differentiation theorem of Pauls \cite{pauls},
which generalizes Kirchheim's metric differentiation theorem \cite{kirchheim}:

\begin{theorem}
\label{thm-pauls}
For almost every  
$g\in G$, rescalings of $\rho$
at $g$ converge on uniformly on compact subsets of $\horiz(G)\subset G\times G$
to the left invariant Carnot (pseudo)distance $\al:G\times G\ra[0,\infty)$
induced 
by some Finsler semi-norm on the horizontal space.   In other words, if
$K\subset \horiz(G)$ is compact, then 
\begin{equation}
\label{eqn-metricrescalings}
\frac{1}{\la}\,s_\la^*(\ell_g^*\rho)\restr_K\stackrel{C^0}{\lra}\;\al\restr_K\,\quad
\mbox{as}\quad \la\ra 0\,.
\end{equation}

\end{theorem}

In actuality, suitably 
formulated,  metric  differentiation
holds whenever the domain is 
any PI space; see \cite{ckmetdiff}.

One may refine the conclusion somewhat by making use of ultralimits, which 
have been used frequently in geometric group theory 
see \cite{asyinv,kleinerleeb}, or earlier  in the  Banach space literature \cite{dacuna,heinrich,heinrichmankiewicz}.
If $g$ is as in the theorem above, and $\{\la_k\}\subset(0,\infty)$ is a 
sequence tending to zero, then $f$ defines a sequence of uniformly
Lipschitz maps $(\frac{1}{\la_k}G,g)\stackrel{f}{\lra}(\frac{1}{\la_k}Y,f(g))$
between pointed metric spaces.  The ultralimit of this
sequence is a Lipschitz mapping $f_\om:G_\om\ra Y_\om$,
where $G_\om$ and $Y_\om$ are ultralimits of $G$ and $Y$, 
respectively.  Up to isometry, $G_\om$ may be identified
with $G$ itself,  while a theorem of 
Kakutani \cite{kakutani} implies that when $Y=L^1$, then
$Y_\om=L^1_\om$  is isometric to an
$L^1$ space for some (typically not $\si$-finite) measure.
This gives:
\begin{corollary}
\label{cor-paulspluskakutani}
If the rescaled (pseudo)distance functions
 in (\ref{eqn-metricrescalings}) converge 
uniformly on compact subsets of $G\times G$ to a limiting pseudo-distance
$\rho_{\infty}$, then $\rho_{\infty}$ is the metric induced by
a map $f_\om:G\ra L^1$.
\end{corollary}

\bigskip
\subsection{$L^1$ metrics and cut metrics}
\label{subsec-cutmetrics}

We refer the reader to \cite{ckbv} for more discussion of the 
material reviewed in this subsection.   We are using a slightly 
different setup here, working
with $L^1_{\loc}$ rather than $L^1$, but the adaptation to this setting
is straightforward.

We let $(X,\mu)$ denote a locally compact metric measure space, where 
$\mu$ is a Borel measure  which is finite on compact subsets of $X$.

A {\bf cut} in $X$ is an equivalence class of measurable subsets, where
two subsets $E,\,E'$ are equivalent if their symmetric difference has 
measure zero. We let $\cut(X)$ denote the collection of cuts in $X$.  
We may view $\cut(X)$ as a subset of  $ L^1_{\loc}(X)$, by identifying
a cut $E\in \cut(X)$ with its characteristic function $\chi_E\in L^1_{\loc}(X)$.
We will endow $\cut(X)$ with the topology induced by $L^1_{\loc}(X)$ via
this embedding. 

A {\bf cut measure} on $X$  is a Borel measure $\Si$ on $\cut(X)$
such that 
$$
\int_{\cut(X)}\;\mu(E\cap K)\;d\Si(E)<\infty
$$
for every compact subset $K\subset X$. 

For every cut measure $\Si$, there is a {\bf tautological} $\Si\times\mu$-measurable
function $\Phi:\cut(X)\times X\ra \{0,1\}$ such that for $\Si$-a.e. cut
$E$, we have $\Phi(E,x)=\chi_E(x)$  for $\mu$-a.e. $x\in X$; this function
is unique, up to sets of measure zero by Fubini's theorem.
For such a function $\Phi$, if $x\in X$, $E\in \cut(X)$,
we let  $\Phi_x=\Phi(\cdot,x)$ and $\Phi_E=\Phi(E,\cdot)$.

If $\Si$ is a cut measure with tautological function $\Phi$, then we obtain
an $L^1_{\loc}$ mapping $X\ra L^1(\cut(X),\Si)$ by sending 
$x\in X$ to $\Phi_x$.  In particular, there is a full measure subset
$Z\subset X$ such that if $x_1,\,x_2\in Z$, then $\Phi_{x_i}$ is 
$\Si$-integrable and so we obtain a (pseudo)distance 
$$
d_{\Si}(x_1,x_2)=\|\Phi_{x_1}-\Phi_{x_2}\|_{L^1(\cut(X),\Si)}\,,
$$
which is the {\bf cut metric} associated with the cut measure $\Si$. 
Modulo changing  $Z$ by a set of measure zero, the cut metric is independent
of the choice of tautological function $\Phi$.

A cut $E\in \cut(X)$ defines an {\bf elementary cut metric $d_E:X\times X\ra [0,\infty)$},
where $d_E(x_1,x_2)=|\chi_E(x_1)-\chi_E(x_2)|$.  Since $\Phi_E=\chi_E$ for
$\Si$-a.e. $E\in \cut(X)$, we may view the cut metric $d_{\Si}$ as a superposition
of elementary cut metrics:
\begin{equation}
\label{eqn-superposition}
d_{\Si}(x_1,x_2)=\int_{\cut(X)}\;
|\Phi_{x_1}(E)-\Phi_{x_2}(E)|\,d\Si(E)
=\int_{\cut(X)}\;d_E(x_1,x_2)\,d\Si(E)
\end{equation}

Notice that above discussion of cut measures and associated cut metrics
makes perfect sense for signed measures.   This leads to the notion of a {\bf signed
cut measure}, and the associated {\bf cut metric} which is still
given by  (\ref{eqn-superposition}), except that it  may  take negative
values.   
Signed cut measures will appear in Section \ref{sec-injectivity}.

Now let $f:(X,\mu)\ra L^1(Y,\nu)$
be an $L^1_{\loc}$ mapping, where $(Y,\nu)$ is a $\si$-finite measure
space,  and let $\rho=f^*d_{L^1(Y,\nu)}$ be the
pullback distance, $\rho(x_1,x_2)=\|f(x_1)-f(x_2)\|_{L^1(Y,\nu)}$.
Then $\rho$ arises from a cut measure:

\bigskip\begin{theorem}
\label{thm-cutmetricproperties}
There is a cut measure $\Si$ such that for any tautological 
function $\Phi:\cut(X)\times X\ra \{0,1\}$ as above, 
there is a full measure subset $Z\subset X$ such that if
$x_1,\,x_2\in Z$, then 
$\rho(x_1,x_2)=d_{\Si}(x_1,x_2)$.

\end{theorem}

\bigskip
\begin{remark}

In some respects a more natural setting for the material
in this section would be a $\si$-finite measure space equipped with
an exhaustion $X_1\subset X_2\subset\ldots$ by finite measure subsets.
Since our applications  only involve 
 locally compact metric measure
spaces, we have chosen this setting.
\end{remark}

\bigskip
\section{Monotonicity  and geodesic maps to $L^1$}
\label{sec-monotonegeodesic}

In this section, we show that geodesic maps to $L^1$
may characterized by a monotonicity property of the cuts
in the support of the cut measure.

\subsection*{Geodesic maps from $\R$ to $L^1$}
We begin with the following observation:

\begin{lemma}
Suppose $f=(f_1,\ldots,f_n):\R\ra \ell^1(\R^n)$ is a continuous map.
Then  $f$ is a weakly monotonic parametrization of a geodesic
in $\ell^1(\R^n)$ if and only if each component $f_i:\R\ra\R$
is weakly monotonic.
\end{lemma}
\begin{proof}
For any  $a\leq b\leq c\in \R$,
$$
\|f(a)-f(c)\|=\sum_i\;|f_i(a)-f_i(c)|
\;
\leq\;\sum_i\;\left(|f_i(a)-f_i(b)|+|f_i(b)-f_i(c)|\right)
$$
$$
=\|f(a)-f(b)\|+\|f(b)-f(c)\|\;.
$$
Therefore we have equality if and only if $|f_i(a)-f_i(c)|=|f_i(a)-f_i(b)|+|f_i(b)-f_i(c)|$
for all $i$.  The lemma follows.
\end{proof}

It is natural to ask for an equivalent characterization in terms of the
associated cut measure.  This leads to:

\bigskip\begin{definition}
\label{def-monotoner}
A cut (or measurable subset) $E\subset\R$ is {\bf monotone} if it is equivalent
to a measurable subset which is connected, and has connected complement.
\end{definition}
Every monotone cut $E\subset\R$ may be represented by the empty set, a ray, or $\R$.
We use the word ``monotone'' for this condition, because  monotone functions have
monotone sublevel/superlevel sets, and  the characteristic function of a measurable
set is essentially monotone if and only if the set is monotone.

Given a distance function $\al$ on a subset $\{a,b,c\}\subset\R$,
where $a\leq b\leq c$, the {\bf excess} of $\al$ is the quantity
$\excess(\al)\{a,b,c\}=\al(a,b)+\al(b,c)-\al(a,c)\geq 0$.  Note that if
$d_E$ is the elementary cut metric associated with a measurable
subset $E\subset \R$,
then $E$ is monotone
if and only if $\excess(d_E)\{x_1,x_2,x_3\}=0$ for $\l^3$-a.e. triple
$(x_1,x_2,x_3)$.  To see this, observe that if $a<b\in \R$ lie in the support
of $E$ (respectively $\R\setminus E$), then $(a,b)\setminus E$ (respectively
$(a,b)\cap E$) has measure zero.

\bigskip\begin{lemma}
\label{lem-geodesicmonotoner}
Suppose $f:\R\ra L^1$ is an $L^1_{\loc}$ mapping.  Then the following
are equivalent:
\begin{enumerate}
\item There is a full measure subset $Z\subset \R$ such that
if $z_1,\,z_2,\,z_3\in Z$ and $z_1\leq z_2\leq z_3$, then 
$$\|f(z_1)-f(z_3)\|_{L^1}=\|f(z_1)-f(z_2)\|_{L^1}+\|f(z_2)-f(z_3)\|_{L^1}\,.
$$
\item If $\Si$ is the cut measure associated with $f$, then  $\Si$-a.e. 
cut $E$ is monotone.
\end{enumerate}
\end{lemma}
\begin{proof}

Let $\Si$ be the cut measure on $X$ guaranteed by
Theorem \ref{thm-cutmetricproperties}, and let $\rho:\R\times\R\ra [0,\infty)$
be the pullback of the distance on $L^1$ by $f$, so 
$$
\rho=\int_{\cut(\R)}\,d_E\,d\Si(E)\,
$$
where $d_E$ is the elementary cut metric associated with $E$.
  Then 
$$
\int_{\R^3}\;\excess(\rho)\{z_1,z_2,z_3\}\,d\l^3(z_1,z_2,z_3)
$$
$$
=\int_{\R^3}\;\excess\left(\int_{\cut(\R)}\;d_E\,d\Si(E)\right)\{z_1,z_2,z_3\}\,d\l^3(z_1,z_2,z_3)
$$$$
=\int_{\R^3}\;\left(
\int_{\cut(\R)}\;\excess(d_E)\{z_1,z_2,z_3\}\,d\Si(E)\right)\,d\l^3(z_1,z_2,z_3)\,.
$$
By Fubini's theorem, it follows that the vanishing of the quantity above is equivalent
to either (1) or (2). 
\end{proof}

\bigskip
\subsection*{$L^1$-mappings which are geodesic along a family of curves}

We now consider an $L^1_{\loc}$-mapping $f:(X,\mu)\ra L^1$
with associated cut measure $\Si$, 
where $(X,\mu)$ is a locally compact space and $\mu$ is 
a Radon measure.  We would
like to examine the implications for $\Si$ when 
$f$ is geodesic along certain curves in $X$.  However, both 
 cuts and  $L^1$ mappings are only 
defined up to sets of measure zero, and since
curves typically have measure zero, this relation only becomes meaningful
when we consider generic curves belonging to a sufficiently rich family.
We formalize this as follows.

 Let  
$\P$ be a locally compact measure space equipped with a Radon
measure $\pi$, and $\Ga:\R\times \P\ra X$ be 
a continuous  map  such that the pushforward measure
satisfies $\Ga_*(\l\times\pi)\leq C \mu$ for some $C\in \R$. 
 If $p\in \P$, we denote the map
$t\mapsto \Ga(t,p)$ by $\Ga_p$.

\bigskip\begin{definition}
\label{def-gamonotone}
A measurable subset $E\subset X$ is {\bf $\Ga$-monotone} if
for $\pi$-a.e. $p\in \P$,  the inverse image $\Ga_p^{-1}(E)\subset \R$
is measurable and monotone.  By Fubini's theorem, this property
is shared by all measurable subsets representing the same cut, and
hence we may speak of {\bf $\Ga$-monotone cuts}.
\end{definition}

\bigskip\begin{proposition}
\label{prop-geodesicmonotone}
The following are equivalent:
\begin{itemize}
\setlength{\itemsep}{.75ex}%
\setlength{\parskip}{.75ex}%
\item For $\pi$-a.e. $p\in \P$, the composition $f\circ \Ga_p:\R\ra L^1$
is a geodesic map, i.e. it satisfies the conditions of Lemma \ref{lem-geodesicmonotoner}.
\item  $\Si$-a.e. cut $E$
is $\Ga$-monotone.

\end{itemize}
\end{proposition}
\begin{proof}
In brief, after sorting out the behavior of cut metrics under
composition of maps and slicing, this follows from the previous
lemma.

Let $\Si$, the map $\Phi:\cut(X)\times X\ra \{0,1\}$, and $Z\subset X$ be as in Theorem  
\ref{thm-cutmetricproperties}, and let $\rho=f^*d_{L^1(Y,\nu)}$
be the pullback distance.

 For $\pi$-a.e. $p\in\P$,  the map $\Ga:\R\times\P\ra X$ induces a map
$\cut(\Ga_p):\cut(X)\ra \cut(\R)$ given by $\cut(\Ga_p)(E)=\Ga_p^{-1}(E)$;
 pushing $\Si$ forward 
under $\cut(\Ga_p)$ we get an $L^1_{\loc}$ cut  measure
$\Si_p$ on $\cut(\R)$.   For such $p$, we may choose 
a $(\Si_p\times\l)$-measurable function $\hat\Phi^p:\cut(\R)\times\R\ra \{0,1\}$
such that for $\Si_p$-a.e. $E\in \cut(\R)$ we have $\hat\Phi^p(E,\cdot)=\chi_E$.
Also, for $\pi$-a.e. $p\in \P$, we have
a well-defined $L^1_{\loc}$-map $f_p=f\circ\Ga_p:\R\ra L^1(Y,\nu)$,
and a pullback distance $\rho_p=f_p^*d_{L^1(Y,\nu)}$.

\bigskip
\begin{lemma}
\label{lem-twosame}
For $\pi$-a.e. $p\in\P$, and $\Si$-a.e. $E\in \cut(X)$,
$$
\Phi(E,\Ga_p(t))=\hat\Phi^p(\cut(\Ga_p)(E),t)
$$ for $\l$-a.e. $t\in\R$.
\end{lemma}
\begin{proof}
By Fubini's theorem and the defining properties of $\Phi$ and $\hat\Phi^p$, 
for $\pi$-a.e. $p\in \P$, and $\Si$-a.e. $E\in \cut(X)$, 
$$
\Phi(E,\Ga(p,t))=\chi_{\cut(\Ga_p)(E)}(t)=\hat\Phi^p(\cut(\Ga_p)(E),t)
$$
for $\l$-a.e. $t\in \R$.
\end{proof}

\bigskip
Now for $\pi$-a.e. $p\in \P$, there is a full measure set
$T_p\subset\R$ such that $\Ga_p(T_p)\subset Z$, and therefore
for every $t_1,\,t_2\in T_p$ we get
$$
\rho_p(t_1,t_2)=\int_{\cut(X)}\;|\Phi_{\Ga_p(t_1)}(E)-\Phi_{\Ga_p(t_2)}(E)|\,d\Si(E)
$$
$$
=\int_{\cut(\R)}\;|\hat\Phi^p_{t_1}(E)-\hat\Phi^p_{t_2}(E)|\,d\Si_p(E)
$$
by Lemma \ref{lem-twosame}.  Lemma \ref{lem-geodesicmonotoner}
implies that $\rho_p$ satisfies the conditions of the lemma for 
$\pi$-a.e. $p\in\P$ if and only if $\Phi(E,\Ga_p(\cdot))$
is the characteristic function of a monotone set for $\pi$-a.e.
$p\in\P$ and $\Si$-a.e. $E\in \cut(X)$. 

\end{proof}

\bigskip
Next we apply Proposition  \ref{prop-geodesicmonotone} to  a Carnot group $G$.  
We recall that $\L(G)$ denotes the family of (horizontal) lines in $G$; we
let 
$\P$ be the family of unit speed parametrized lines in $G$.  Here
$\L(G)$ and $\P$  have natural  smooth structures
such that   the
tautological map $\P\ra \L(G)$ is a smooth fibration with fibers
diffeomorphic to the Lie group $\isom(\R)$.  We endow $\P$
and the space of line $\L(G)$ with smooth measures with positive
density.  If $\Ga:\R\times\P\ra G$
is defined by $\Ga(t,p)=p(t)$, then Fubini's theorem implies that
a measurable set $E\subset G$ is $\Ga$-monotone if and only if
the intersection $E\cap L$ is a monotone subset of $L\simeq \R$
for almost every $L\in \L(G)$.

\begin{definition}
\label{def-monotoneg}
A measurable subset (or cut)  $E\subset G$ is {\bf monotone}
if $E\cap L$ is a monotone subset of $L\simeq \R$ for almost every
$L\in \L(G)$.
\end{definition}
With this definition, Proposition  \ref{prop-geodesicmonotone} yields:

\begin{corollary}
Let $f:G\ra L^1$ be an $L^1_{\loc}$-mapping with associated 
cut measure $\Si$,  such that $f\restr_L$
is a geodesic map for almost every $L\in \L(G)$.  Then 
 $\Si$-a.e.  $E\in \cut(G)$ is monotone. 
\end{corollary}

In Section \ref{sec-monotone} we will show that 
nontrivial monotone
subsets of $\H$ are  half-spaces, modulo sets of measure zero.

\bigskip\bigskip\section{The classification of precisely monotone sets}
\label{sec-preciselymonotone}

Throughout this section $\D E$ will denote the topological frontier (boundary)
of a subset $E$.

In Section \ref{sec-monotone} we will classify monotone subsets of the Heisenberg
group.   Before doing this,  we
first consider the easier task of classifying 
precisely monotone sets:

\begin{definition}
Let $X$ be either $\R^n$ or $\H$.
A subset $E\subset X$ is {\bf precisely monotone} if
 every line $L\in \L(X)$ intersects both $E$ and its complement in
a connected set.
\end{definition}
Thus in the $\R^n$ case, a precisely monotone set is a convex
set with convex complement.

\begin{lemma}
If $E\subset\R^n$ is  precisely monotone, then 
either $E=\emptyset$, $E=\R^n$, or
$C\subset E\subset \ol{C}$ for some (open)
half-space $C\subset \R^n$.
\end{lemma}
This follows immediately by looking at a supporting half-space
for $E$, assuming both $E$ and its complement are nonempty.

We now focus on precisely monotone subsets of the Heisenberg
group:

\begin{theorem}
\label{thm-preciselymonotone}
If $E\subset\H$ is a precisely monotone subset,
then either
$E=\emptyset$, $E=\H$, or
 $C\subset E\subset \ol{C}$ for some half-space $C\subset\H$.
\end{theorem}

For the remainder of this section, we fix a precisely
monotone subset $E\subset \H$, and let $E^c=\H\setminus E$
be its complement.   Note that a subset of $\H$ is
precisely monotone if and only if its complement is precisely monotone,
so the roles of $E$ and $E^c$ will be symmetric throughout.

The proof will proceed in the following steps:

\begin{enumerate}
\setlength{\itemsep}{.75ex}%
\setlength{\parskip}{.75ex}%
\item Lemma \ref{lem-morethanonepoint}: If  $L\in \L(\H)$ 
contains more than one
point of $\D E$, then $L\subset \D E$.
\item Lemma \ref{lem-bdyeunion}: $\D E$ is a union of lines.
\item Lemma \ref{lem-inplaneorh}: Either $\D E$ is contained 
in a plane, or 
$\D E=\H$.
\item Lemma \ref{lem-noth}: The case $\D E=\H$ does not occur.
\item Lemma \ref{lem-deplane}: If $\D E$ is nonempty, then it is a plane
and $C\subset E\subset\bar C$, where $C$ is a connected
component of $\H\setminus \D E$.
\end{enumerate}

We now proceed with the steps of the proof.

\bigskip
 If 
a monotone (or more generally convex) subset $Y\subset \R^n$ contains a subset
$\Si$ and a point $p$, then it also contains the cone
over $\Si$ with vertex at $p$.  We begin with an analogous
statement in the Heisenberg group; it is  more
subtle than the Euclidean case, due to the fact that 
the lines in $\H$ passing through a point $x\in \H$
lie in a horizontal plane, which has empty interior.
To implement the argument, we will use piecewise horizontal
curves.  

\begin{definition}
\label{def-gammas}
For  $x\in \H$, $v_1,\,v_2\in \De$, let 
$\ga(x,v_1,v_2)$ be the unit speed path which starts
at $x$, moves along a horizontal curve in the direction
$v_1$ a distance $|v_1|$, and then along a horizontal
curve in the direction $v_2$ a distance $|v_2|$.  Thus
$\ga(x,v_1,v_2)$  is a broken horizontal line with vertices
$x$, $x\exp(v_1)$, and $x\exp(v_1)\exp(v_2)$.
Using this, we may define
a map $\Ga:\H\times\De^2\ra \H$ by letting
$\Ga(x,v_1,v_2)$ be the other endpoint of $\ga(x,v_1,v_2)$,
i.e. $\Ga(x,v_1,v_2)=\ga(x,v_1,v_2)(|v_1|+|v_2|)$.
For $x\in\H$, we
define $\Ga_x:\De^2\ra \H$ by $\Ga_x(v_1,v_2)=\Ga(x,v_1,v_2)$.
\end{definition}

\begin{lemma}
\label{lem-gapsubmersion}
The map $\Ga$ is smooth.  For all $x$, the map $\Ga_x$
is a submersion near any pair $(v_1,v_2)\in \De^2$
with $v_1+v_2\neq 0$ (recall that we are viewing
$\De$ as a subspace of the Lie algebra of $\H$).
\end{lemma}
\begin{proof}
The smoothness of $\Ga$ is immediate from the smoothness
of the group operation.  

Let $\pi:\H\ra \R^2$ be the abelianization map.
Pick $x\in \H$, $(v_1,v_2)\in \De^2$ such that $v_1+v_2\neq 0$.  
Then $\pi(\Ga_x(v_1,v_2))$
is the point $y\defeq x+\pi(v_1)+\pi(v_2)$.  Define a smooth
path $\eta:\R\ra\De^2$ by $t\mapsto (v_1+tw,v_2-tw)$ where
$w$ is a nonzero vector orthogonal to $v_1+v_2$.  Then 
$\Ga_x\circ\eta$ has a nonzero velocity tangent to the
fiber $\pi^{-1}(y)$ (this follows by using (\ref{eqn-heightequalsenclosedarea})).
  Evidently $D(\pi\circ\Ga_x)(v_1,v_2)$
is onto, which implies that $D\Ga_x(v_1,v_2)$ is onto as well.
\end{proof}

\bigskip
\begin{proposition}
\label{prop-2points}
Suppose $L\in \L(\H)$, $p\in L$, and $\Si\subset E$ is
a surface
intersecting $L$ transversely at a point $q\in L\setminus \{p\}$.
Then:
\begin{enumerate}
\setlength{\itemsep}{.75ex}%
\setlength{\parskip}{.75ex}%
\item If $p\not\in \Int(E^c)$, then  the open segment
$(p,q)\subset L$  is contained in $\Int(E)$.
\item If $p\in E^c$, then the connected component  of $L\setminus\{q\}$
not containing $p$ lies in $\Int(E)$.  
\end{enumerate}
The same statements hold with the roles of $E$ and $E^c$ exchanged.
\end{proposition}

\begin{remark}
The first assertion still holds if one merely assumes that
$E$ is {\bf precisely convex}, i.e. its intersection with
any $L\in\L(\H)$ is connected.
\end{remark}

\begin{proof}
We orient the line $L$ in the direction from $p$ to $q$.
Choose a point $y\in L$ which is separated from $p$ by $q$.
Since $L$ intersects $\Si$ transversely at $q$, any path 
close to the segment $\ol{py}\subset L$  will intersect $\Si$.

Choose $z\in (p,q)$, and define $\bar v\in \De$ by
$z=p\,\exp 2\bar v$.  Thus $z=\Ga_p(\bar v, \bar v)$.

We claim that 
 there is an $\eps>0$
such that if  $x\in E$, $v_1,\,v_2\in \De$
satisfy  $$
\max(d^{\H}(x,p),\|v_1-\bar v\|,\|v_2-\bar v\|)<\eps\,,
$$
then $\Ga_x(v_1,v_2)\in E$.  
To see this,
note that when $\eps$ is sufficiently small,
we may choose $\rho_1\in(1,\infty)$
such that $\ga(x,\rho_1 v_1,0)$ is a segment
ending near $y$, and by precise monotonicity the
subsegment 
$\ga(x, v_1,0)$ lies in $E$.
Similarly,  we
can choose $\rho_2\in (1,\infty)$ such that
$\ga(x,v_1,\rho_2v_2)$ is a path ending
near $y$, so precise monotonicity  implies that 
 $\ga(x,v_1,v_2)=\Ga_x(v_1,v_2)$ is contained in $E$. 

By Lemma
\ref{lem-gapsubmersion}, the map $\Ga_p$
restricts to a submersion on a ball $B\subset \De^2$ centered
at $(\bar v,\bar v)$; by shrinking $B$, we may assume
that it is contained in the set
$\{(v_1,v_2)\in \De^2\mid \|v_i-\bar v\|<\eps\}$. Therefore
 by the implicit function theorem,
if we choose $x\in E\cap B(p,\eps)$
sufficiently close to $p$, then  $\Ga_x$ will map
$B$ onto a neighborhood $U$ of $z$; by the preceding
paragraph we have $U\subset E$.  Since $z$ was an arbitrary
point in $(p,q)$, we have $(p,q)\subset\Int(E)$.

The proof of part (2) is similar, except that one 
considers paths $\ga(p,v_1,v_2)$ which cross $\Si$, and
the component of $\ga(p,v_1,v_2)\setminus \Si$ lying 
on the other side of $\Si$.
\end{proof}

Proposition \ref{prop-2points} implies:
\begin{lemma}
\label{lem-morethanonepoint}
If $L\in \L(\H)$, and $L\cap \D E$ contains more
than one point, then $L\subset \D E$.
\end{lemma}
\begin{proof}
Suppose $p,q\in L\cap \D E$ are distinct points,
and $x\in L\setminus \D E$.  Then either $x$ is in the
interior of $E$, or $E^c$; without loss of generality
we may assume that $x\in \Int(E)$.  Then there is a surface 
$\Si\subset \Int(E)$ intersecting $L$ transversely at
$x$.   By Proposition \ref{prop-2points}, the open segment of $L$
lying between $p$ and $x$ lies in $\Int(E)$; therefore
$x$ lies between $p$ and $q$.  Hence any point 
$y\in L\setminus\{p,q\}$ separated from  $p$ by $q$ belongs to
$\D E$.  Repeating the above reasoning with 
$p$ replaced by some $p'$ between $p$ and $q$ gives
a contradiction.
\end{proof}

\bigskip

\begin{lemma}
\label{lem-bdyeunion}
$\D E$ is the union of the lines
it contains.  
\end{lemma}
\begin{proof}

Pick $x\in \D E$.  

Suppose every line $L\in \L(\H)$
which passes through $x$ intersects $\D E$ only at $x$.
Then the union of the lines passing through $x$ is 
a horizontal plane $P$, and $(P\setminus \{x\})\cap \D E=\emptyset$.
Since $P\setminus\{x\}$ is connected, it follows that
either $P\setminus \{x\}$ is entirely contained in $\Int(E)$
or $\Int(E^c)$; without loss of generality we assume
the former.  Because $x\in \D E$, there
is a sequence $\{x_k\}\subset E^c$ which
converges to $x$.  For each $k$, choose a line $L_k$
passing through $x_k$, and (by precise monotonicity)
a ray $\eta_k\subset L_k\cap E^c$
containing $x_k$.  Then the sequence $\{\eta_k\}$ will
accumulate on some point in $P\setminus \{x\}$ contradicting the fact
that $P\setminus \{x\}\subset \Int(E)$.

\end{proof}

\bigskip
Our next goal is:
\begin{lemma}
\label{lem-skeworparallel}
Suppose $\g\subset\H$ has the property that if $L\in \L(\H)$
and $L\cap \g$ contains more than one point, then $L\subset \g$.
If $\g$ contains either a pair of skew lines or a pair
of parallel lines  with distinct projection, then $\g=\H$.
\end{lemma}

\begin{proof}
Observe that if $\g$ contains a pair of parallel lines with distinct
projections, then by part A of Lemma \ref{lem-geometry} and the hypothesis
on $\g$, there will be a pair of skew lines contained in $\g$; therefore
we may assume that $\g$ contains a pair of skew lines.

We first claim that if $L_1$ and $L_2$ are skew lines contained in 
$\g$, then  the fiber 
$$
\pi^{-1}(\pi(L_1)\cap \pi(L_2))
$$ 
is contained in $\g$.

By Lemma \ref{lem-geometry}, there is a hyperbola $Y\subset\R^2$ with
asymptotes $\pi(L_1)$ and $\pi(L_2)$ such that every tangent line of $Y$
has a unique lift $L\in\L(\H)$ which intersects both $L_1$ and $L_2$, and which is
therefore contained in $\g$.  Thus we can
find a pair of parallel lines $L_3,\,L_4\in \L(\H)$ which intersect both
$L_1$ and $L_2$, such that $\pi(L_3)$ and $\pi(L_4)$ are distinct tangent 
lines of the hyperbola $Y$.  

By Lemma \ref{lem-geometry}
the collection $\C$ of lines  which intersect both $L_3$ and 
$L_4$ contains $L_1$ and $L_2$, and their union contains 
$$
\pi^{-1}(\pi(L_1)\cap \pi(L_2)).
$$
Since 
$\cup_{L\in \C}L\subset \g$, the claim is established.

Note that the hyperbola $Y$ separates $\R^2$ into three
connected components, and let $U$ be the one whose closure 
contains $Y$. Let $V\defeq U\setminus\{L_1\cup L_2\}$. 
Every point in $V$
is the intersection point 
of two tangent lines of $Y$, and the corresponding lifts
will be skew lines contained in $\g$.  Therefore by the claim,
we have $\pi^{-1}(V)\subset \g$.

Now if $x\in \H$, there is an $L\in \L(\H)$ containing $x$ which
passes through $\pi^{-1}(V)$, and such an $L$ will intersect $\g$
in more than one point, forcing $x\in \g$.  Thus  $\g=H$.
\end{proof}

Using this lemma, we get:
\begin{lemma}
\label{lem-inplaneorh}
Either $\D E$ is contained in a plane, or $\D E=\H$.
\end{lemma}
\begin{proof}
Assume that $\D E\neq \H$, but that $\D E$ is not
contained in a plane.  

By Lemma \ref{lem-bdyeunion},
we know that $\D E$ is a union of lines; since $\D E$
is not contained in a plane, it must therefore contain
at least two lines $L_1,\,L_2$.  

  If  $L_1,L_2$ have parallel projection, then $\pi(L_1)=\pi(L_2)$;
otherwise by Lemma \ref{lem-skeworparallel} we would
contradict our assumption that $\D E\neq \H$.
Furthermore,  any third line $L\subset \D E$  must also have 
the same projection,
since otherwise $\D E$ would contain a pair of skew lines or a pair
a parallel lines with distinct projection.  So in this case $\D E$ 
is contained in a vertical plane.

Therefore we may assume that $\D E$ does not contain
distinct lines with parallel projection.
Hence every 
pair of lines contained in $\D E$ must intersect.  Since any
triple of lines which intersect pairwise must have a common
intersection point, it follows
that all lines contained in $\D E$ pass through a single
point $x\in \H$, and hence $\D E$ is contained in a horizontal 
plane.

\end{proof}

\begin{lemma}
\label{lem-noth}
Either $E$ or $E^c$ has nonempty interior; equivalently,
$\D E\neq \H$.
\end{lemma}
\begin{proof}
Let $L_1,L_2\in\L(\H)$ be skew lines.  By Lemma \ref{lem-geometry},
for $i\in \{1,2\}$, we can find open intervals 
$(a_i,b_i)\subset L_i$
and smooth parametrizations $x_i:I\ra (a_i,b_i)$
such that for all $t\in I$, the points $x_1(t)\in L_1$
and $x_2(t)\in L_2$ lie in a line $L_t$.  By monotonicity,
after passing to subintervals if necessary, we may assume that
for $i\in \{1,2\}$, the characteristic function $\chi_E$ is constant on $(a_i,b_i)$,
i.e. it lies entirely in $E$, or entirely in $E^c$.

If $(a_1,b_1)\cup (a_2,b_2)\subset E$ (respectively
$E^c$), then by monotonicity
for all $t\in I$, the segment $[x_1(t),x_2(t)]\subset L_t$ lies in 
$E$ (respectively $E^c$). Hence $\cup_{t\in I}\;(x_1(t),x_2(t))$
is a relatively open subset $U$ of a ruled surface
which is contained in $E$ (respectively $E^c$).   
If $(a_1,b_1)\subset E$ and $(a_2,b_2)\subset E^c$ (or vice-versa), 
then  $L_t\setminus [x_1(t),x_2(t)]$ is a union two rays, one of 
which lies in $E$, and the other lies in $E^c$.  Therefore
in this case, we also obtain a relatively open subset $U$
of a ruled surface, which lies in $E$, or in $E^c$.

Choose a line  $L$ which intersects
the surface $U$ transversely at some point $p$, and
pick $x\in L\setminus U$.  Then Proposition \ref{prop-2points} 
implies that either $E$ or $E^c$
has nonempty interior.

\end{proof}

\begin{lemma}
\label{lem-deplane}
If $E$ and $E^c$ are both nonempty, then $\D E$
is a plane, and $C\subset E\subset\bar C$, where
$C$ is a component of $\H\setminus \D E$.
\end{lemma}
\begin{proof}
By assumption, $\D E$ is nonempty.  By Lemmas
\ref{lem-inplaneorh} and \ref{lem-noth}, it follows
that $\D E$ is contained in a plane $P$.  If $\D E\neq P$, 
then $\H\setminus \D E$ is connected, and hence $\H\setminus \D E$
is contained in $E$, or in $E^c$.  We may therefore
assume that it is contained in $E$.  But then every point 
$x\in P$ other than the center of $P$  (if $P$ is horizontal)
 lies on a line $L$ transverse to $P$, so $L\setminus \{x\}$
is contained in $E$, and by precise monotonicity, we get 
$L\subset E$.  It follows that $E=\H$.  This is a contradiction,
so $\D E=P$.  The assertion that $C\subset E\subset \bar C$
follows from the Jordan separation theorem (or by the elementary 
fact that a horizontal or vertical plane separates into two components).
\end{proof}

\section{The classification of monotone sets}
\label{sec-monotone}

The goal of this section is:

\begin{theorem}
\label{thm-monotonehalfspace}
If $E\subset \H$ is a monotone set, then modulo a
null set, either $E=\emptyset$, $E=\H$, or $E$ is a 
half-space.
\end{theorem}

For the remainder of this section, $E$ will
denote a fixed
monotone set $E\subset\H$.

\subsection*{Measure theoretic preparation}
The proof of Theorem \ref{thm-monotonehalfspace} will
follow the proof of Theorem \ref{thm-preciselymonotone}
closely.  The main difference stems from the fact that
the monotonicity condition only holds for almost every line $L$,
and up to a null set within $L$; this forces 
one to modify the proof of 
Theorem \ref{thm-preciselymonotone} by considering positive
measure families of certain configurations (such as 
piecewise horizontal curves), instead
of single configurations.

We begin with a measure-theoretic replacement for the boundary
and interior.  These were chosen so that the proof  of the classification of 
monotone sets closely parallels the proof for precisely monotone sets.

\begin{definition}
The {\bf support} of a measurable set $E$, denoted
 $\spt(E)$, is the 
support of its characteristic function.
The {\bf (measure-theoretic) boundary} $\dmu E$ of $E\subset\H$
is $\spt(E)\cap \spt(E^c)$, i.e.
the set of points $x\in \H$ such that 
$\min (\mu(B_r(x)\cap E),\mu(B_r(x)\cap E^c))>0$ for all
$r>0$.  The {\bf measure-theoretic interior} $\Intmu(E)$
of $E$  is $\H\setminus\spt(E^c)$.
\end{definition}
Note that $\dmu E$ is a closed set, and $\H\setminus\D E$
is the disjoint union of $\Intmu(E)$ and $\Intmu (E^c)$.
Similar defintions apply to subsets of $\R$.
Also, a subset of $\R$ is monotone iff its measure-theoretic
boundary contains at most one point.

\begin{definition}
\label{def-monotonethings}
A line $L\in\L(\H)$ is {\bf monotone} if
$E\cap L$ is a monotone subset of $L\simeq \R$.  A pointed line $(L,p)$
is {\bf monotone} if $p\in L\in\L(\H)$, the line $L$
is monotone, and either  [$p$ belongs to $E$ and lies in the measure-theoretic
interior of $E\cap L$ (relative to $L$)], or 
[$p$ belongs to $E^c$ and lies in the measure-theoretic interior of
$E^c\cap L$ (relative to $L$)].
  A direction
$v\in \De\setminus\{0\}$ is {\bf monotone} if almost every line
tangent to $v$ is monotone.
\end{definition}

Note that $L$ is monotone if $L\cap E$ is a measurable
subset of $L$, and the measure theoretic
boundary of $E\cap L$  (in $L$) contains at most one
point.

\begin{lemma}

\mbox{}
\begin{enumerate}
\setlength{\itemsep}{.75ex}%
\setlength{\parskip}{.75ex}%
\item
\label{item-directionmonotone} Almost every direction $v\in \De$ is monotone.
\item 
\label{item-pointedlinemonotone}
 Almost every pointed line is monotone. 
\end{enumerate}
\end{lemma}
\begin{proof}
(\ref{item-directionmonotone})
Let ${\mathcal P}(\De)$ denote the projectivization of the horizontal
space $\De$.  Then there is a smooth fibration $\L(\H)\ra {\mathcal P}(\De)$
which sends a line $L$ tangent to the direction $[v]\in{\mathcal P}(\De)$ to
$[v]$, 
whose fibers are the lines parallel to a given direction.  Therefore
by Fubini's theorem, for almost every $[v]\in {\mathcal P}(\De)$, the set
of monotone elements in the fiber over $[v]$
is a measurable subset of full measure.  

(\ref{item-pointedlinemonotone})  The space of pointed lines fibers
over $\L(\H)$, with fiber $\R$.  Since there is a full measure set
$Y\subset\L(\H)$ consisting of monotone lines, and almost every point $p$
in the fiber over a monotone line $L$ yields a monotone pointed line
$(L,p)$, the statement follows from Fubini's theorem.
\end{proof} 

If $x\in \H$, $v\in \De\setminus \{0\}$,
we will use $L_{x,v}$ to denote the line passing through 
$x$ tangent to the horizontal left invariant vector vield 
$v$.  

\begin{lemma}
\label{lem-aetriplemonotone}
For a.e. triple $(x,v_1,v_2)\in \H\times\De\times\De$,
the pairs $(L_{x,v_1},x)$, $(L_{x,v_1},x\exp v_1)$,
$(L_{x\exp v_1,v_2},x\exp v_1)$, and 
$(L_{x\exp v_1,v_2},x\exp v_1\exp v_2)$ are monotone.
\end{lemma}
\begin{proof}
Suppose $v_1,v_2\in \De\setminus\{0\}$ are monotone directions.
Then by Fubini's theorem, the set $M_i\subset \H$
of points $w\in \H$
such that $(L_{w,v_i},w)$ is monotone has full measure.
Since the maps $x\mapsto x\exp v_1$, $x\mapsto x\exp v_1\exp v_2$
are diffeomorphisms, it follows that for a.e. $x\in \H$, 
$x\in M_1$, $x\exp v_1\in M_1$, $x\exp v_1\in M_2$, and 
$x\exp v_1\exp v_2\in M_2$.  

Since a.e. direction $v\in \De$ is monotone, the lemma follows.
\end{proof}

\begin{definition}
\label{def-admissiblefamily}
Let $X$, $M$, and $A$ be smooth manifolds.
An {\bf admissible  family 
of submanifolds in $X$} is a smooth submersion $\Phi:M\times A\ra X$
such that $\Phi(\cdot,\al)$ is an embedding for all $\al\in A$.
We let $M_\al$ denote the image of 
$\Phi\restr_{M\times\{\al\}}$, and refer simply to the 
resulting family of submanifolds
$\{M_\al\}_{\al \in A}$.
\end{definition}
 
\begin{lemma}
\label{lem-transversefamily}
Let $\{M_\al\}_{\al\in A}$, $\{N_\be\}_{\be\in B}$
be two admissible families of manifolds in a smooth manifold
$X$, and $S\subset X$ be a measurable subset.  We equip 
$A$, $B$,  $X$, and each of the submanifolds $M_\al$, $N_\be$
 with smooth measures with positive
density.  Assume that
\begin{itemize}
\item  For all $\al\in A$, $\be\in B$,
the submanifolds $M_\al$ and $N_\be$ intersect
transversely in a single point.
\item For a positive measure set $A_1\subset A$, for all 
$\al\in A_1$ the
submanifold $M_\al$ interesects $S$ in a full measure
subset of $M_\al$.
\end{itemize}
Then for a full measure set of $\be\in B$, the submanifold
$N_\be$ intersects $S$ in positive measure subset of $N_\be$.
\end{lemma}
\begin{proof}
Let $I:A\times B\ra X$ be the map which sends 
$(\al,\be)\in A\times B$ to the unique intersection point
of $M_\al$ and $N_\be$.  By the transversality assumption and the
definition of admissible families of submanifolds, the map
$I$ is a smooth submersion from $A\times B$ onto an open subset
of $X$.  Also, for each $\al\in A$ (respectively $\be\in B$),
the restriction of $I$ to $\{\al\}\times B$  (respectively
$A\times\{\be\}$) is a  submersion onto a relatively
open subset of $M_\al$ (respectively $N_\be$).  

Let $\hat S\defeq I^{-1}(S)$.

Suppose $\al\in A_1$.  Since 
$I\restr_{\{\al\}\times B}:
\{\al\}\times B\ra M_\al$ is a submersion, it follows that
$\hat S\cap(\{\al\}\times B)$ has full measure in 
$\{\al\}\times B$.

We claim that there is a full measure subset $B_1\subset B$
such that for every $\be\in B_1$, 
the intersection $\hat S\cap (A\times\{\be\})$ has positive
measure in $A\times\{\be\}$.  To see this, note that otherwise
there would be a positive measure subset $B_0\subset B$
such that for every $\be\in B_0$, the intersection 
$\hat S\cap (A\times \{\be\})$ has zero measure in 
$A\times\{\be\}$.  But then by Fubini's theorem,
$\hat S\cap (A\times B_0)$ has measure zero, which contradicts
the fact that $\hat S\cap(A\times B_0)$ intersects a positive
measure set of fibers $\{\al\}\times B$ in a set of positive
measure.

Now for every $\be\in B_1$, the intersection $S\cap N_\be$
has positive measure in $N_\be$, since its inverse image
under $I\restr_{A\times \{\be\}}$ has positive measure
in $A\times\{\be\}$. 

\end{proof}

\subsection*{The proof of Theorem \ref{thm-monotonehalfspace}}

With our measure-theoretic preparations complete, we will now 
prove
the theorem using essentially the same outline as the proof of 
Theorem \ref{thm-preciselymonotone}:

\begin{enumerate}
\setlength{\itemsep}{.75ex}%
\setlength{\parskip}{.75ex}%
\item If  $L\in \L(\H)$ 
contains more than one
point of $\dmu E$, then $L\subset \dmu E$.
\item  $\dmu E$ is a union of lines.
\item  Either $\dmu E$ is contained 
in a plane, or 
$\dmu E=\H$.
\item Lemma \ref{lem-noth}: The case $\dmu E=\H$ does not occur.
\item Lemma \ref{lem-deplane}: If $\dmu E$ is nonempty, then it is 
a plane
and  $C\subset E\subset\bar C$ modulo sets of measure zero, where $C$ is a connected
component of $\H\setminus \dmu E$.
\end{enumerate}

\bigskip
\begin{proposition}
\label{prop-2pointsmeasurable}
Suppose $L\in \L(\H)$, $p\in L$, and $\{\Si_\al\}_{\al\in A}$
is an admissible family of surfaces.  Assume that
\begin{itemize}
\setlength{\itemsep}{.75ex}%
\setlength{\parskip}{.75ex}%
\item For all $\al\in A$, the surface $\Si_\al$ intersects
$L$ transversely in a single point.
\item There is a 
measurable
subset $A_1\subset A$
such that for every a.e. $\al\in A_1$, the surface $\Si_\al$ 
intersects
$E$  in a set of full measure in $\Si_\al$.
\item For some $q\in L\setminus\{p\}$, there is an 
$\al_0\in \spt(A_1)$
 such that 
$\Si_{\al_0}\cap L=\{q\}$.
\end{itemize}
Then:
\begin{enumerate}
\setlength{\itemsep}{.75ex}%
\setlength{\parskip}{.75ex}%
\item If $p\in\spt(E)$, 
then  the open segment
$(p,q)\subset L$  is contained in $\Intmu(E)$ .
\item If $p\in \spt(E^c)$, then connected component  of 
$L\setminus\{q\}$
not containing $p$ lies in $\Intmu(E)$.  
\end{enumerate}
The same statements hold with the roles of $E$ and $E^c$ exchanged.
\end{proposition}
\begin{proof}
(Compare the proof of Proposition \ref{prop-2points})
\quad Pick $z\in (p,q)\subset L$, where $z=p\exp 2\bar v$, 
$\bar v\in \De$.  Hence $z=\Ga_p(\bar v,\bar v)$.

Choose $\eps\in(0,\infty)$,
and let $\g$ be the set of triples 
$(x,v_1,v_2)\in \H\times\De\times\De$
such that $\max\{d^\H(x,p),\|v_1-\bar v\|,\|v_2-\bar v\|\}<\eps$.

Using the reasoning from the proof Proposition \ref{prop-2points},
to prove (1) of Proposition \ref{prop-2pointsmeasurable},
it suffices to show that when $\eps$ is sufficiently small,
for almost every $x\in B_\eps(p)\cap E$, the point
$\Ga_x(v_1,v_2)$ belongs to $E$.   To establish this, we
need the following:

\begin{lemma}
\label{lem-aetriplesigma}
For a.e. triple $(x,v_1,v_2)\in\g$, the following statements
hold:
\begin{enumerate}
\item The following pairs are monotone:
$(L_{x,v_1},x)$, $(L_{x,v_1},x\exp v_1)$,
$(L_{x\exp v_1,v_2},x\exp v_1)$, and 
$(L_{x\exp v_1,v_2},x\exp v_1\exp v_2)$.
\item There is a  point $w_1\in L_{x,v_1}\cap E$, 
 close
to $q$, such that the pair
$(L_{x,v_1},w_1)$ is monotone.
\item  There is a  point 
$w_2\in L_{x\exp v_1,v_2}\cap E$ close
to $q$, such that the pair
 $(L_{x\exp v_1,v_2},w_2)$ is monotone.
\end{enumerate}
\end{lemma}
\begin{proof}
The first assertion follows from Lemma 
\ref{lem-aetriplemonotone}, so we focus on (2) and (3).

Pick $\de\in (0,\infty)$. We may shrink the surfaces
$\{\Si_\al\}_{\al\in A}$ and choose a small neighborhood $B$ of $L$
in $\L(\H)$, such that every $L'\in B$ intersects every
surface $\Si_\al$ transversely in a single point lying
in $B_\de(q)$.  

By Lemma \ref{lem-transversefamily},
for a.e. $L'\in B$, the intersection $L'\cap E\cap B_\de(q)$
has positive measure in $L'$.  Therefore, if $\eps$
is sufficiently small, for a.e.
$v_1\in B_\eps(\bar v)$, the direction $v_1$ is monotone, and for
a.e. $x\in B_\eps(p)$ the pair $(L_{x,v_1},x)$
is monotone and the line $L_{x,v_1}$ intersects
 $E\cap B_\de(q)$ in a positive measure set.  It follows
that for some $w_1\in E\cap B_\de(q)\cap L_{x,v_1}$
the pair $(L_{x,v_1},w_1)$ is monotone. This implies that (2)
holds for a.e. triple $(x,v_1,v_2)$. 

By the same token, for $\eps'$ sufficiently
small,  there is a full measure subset
$M\subset  B_{\eps'}(p\exp(\bar v))$ such that for every $y\in M$ and a.e.
$v_2\in B_\eps(\bar v)$, there is a 
$w_2\in E\cap B_\de(q)\cap L_{y,v_2}$ such that 
$(L_{y,v_2},y)$ and $(L_{y,v_2},w_2)$ are both monotone.
The map $(x,v_1)\mapsto x\exp v_1$ being a submersion,
if $\eps$ is sufficiently small, we conclude that
for a.e. $x\in B_\eps(p)$, $v_1\in B_\eps(\bar v)$, the
point $x\exp v_1$ lies in $M$, and hence for a.e. 
$v_2\in B_\eps(\bar v)$ the desired point 
$w_2\in L_{x\exp v_1,v_2}$ exists.

\end{proof}
To prove (2) of Proposition \ref{prop-2pointsmeasurable},
we combine Lemma \ref{lem-aetriplesigma} and the argument
of (2) in Proposition \ref{prop-2points}.

\end{proof}

\bigskip
\bigskip
Step (1) of the outline
follows, by using Proposition \ref{prop-2pointsmeasurable} in place
of Proposition \ref{prop-2points}.

Steps (2), (3), and (5) then follow using essentially the same 
reasoning in the precisely monotone case.

To implement step (4) for monotone sets, it suffices to
produce an admissible family of surfaces $\{\Si_\al\}_{\al\in A}$
as in the statement Proposition \ref{prop-2pointsmeasurable}
(or the version with $E^c$ replacing $E$).
To that end we have the following:

\begin{lemma}
\label{lem-ael1l2}
For a.e. pair $(L_1,L_2)\in \L(\H)\times\L(\H)$,
a.e. line $L$ intersecting both $L_1$ and $L_2$
is monotone,  and the pairs $(L_i,L_i\cap L)$ and
$(L,L\cap L_i)$
are monotone for $i\in \{1,2\}$.
\end{lemma}
\begin{proof}
Consider the manifold $\F$ of pointed lines $(L,p)$, and the
two submersions $\pi_{\L}:\F\ra \L(\H)$ and $\pi_{\H}:\F\ra \H$,
where $\pi_{\L}(L,p)=L$ and $\pi_{\H}(L,p)=p$.  We know that
a.e. $(L,p)\in \F$ is monotone, because a.e. $L\in \L(\H)$
is monotone, and  a.e. pointed line in the fiber  $\pi_{\L}^{-1}(L)$
over a monotone line $L$, is monotone.

To any pair of skew lines $(L_1,L_2)$, we may associate two 
$1$-dimensional submanifolds 
$M^1_{L_1,L_2},M^2_{L_1,L_2}\subset\F$, namely $M^i_{L_1,L_2}$
is the set of pointed lines $(L,p)$ where $L$ intersects both 
$L_1$ and $L_2$, and $p=L\cap L_i$.  The implicit function
theorem implies that for $i\in \{1,2\}$, the family 
$$
\{M^i_{L_1,L_2}\mid (L_1,L_2)\;\mbox{are skew lines}\}
$$
is admissible in the sense of Definition 
\ref{def-admissiblefamily}.

It follows that for a.e. pair $(L_1,L_2)$ of skew lines,
a.e. pointed line $(L,p)$ on $M^i_{L_1,L_2}$ is monotone.   Since
a.e. line is monotone, the lemma follows.
\end{proof}

Using the lemma, we may imitate the construction of Lemma
\ref{lem-noth} using a family of pairs, in the measurable
setting. This yields  for a.e. $(L_1,L_2)$, open intervals
$(a_i,b_i)\subset L_i$ and smooth parametrizations 
$x_i:I\ra (a_i,b_i)$ as in Lemma \ref{lem-noth}, which 
vary measurably with $(L_1,L_2)$.  By Lusin's
theorem, after passing to a subset
of nonzero measure, we may arrange that the interval
$(a_i,b_i)$ and maps $x_i$ varying continuously with $(L_1,L_2)$.
Then the rest of the lemma may be implemented as in Lemma
\ref{lem-noth}.

\bigskip
\section{The proof  of a weak version of Theorem \ref{thm-maindiff}}
\label{sec-proofweak}

In this section we prove a weak version of the main theorem, which is
enough to imply the nonexistence of bi-Lipschitz embeddings $\H\ra L^1$.
We include this result here because it follows easily from the work
we have done so far, and avoids the harmonic analysis
in Section \ref{sec-injectivity}.

\begin{theorem}
\label{thm-weakdiff}
Let $Z$ be the infinitesimal generator of $\cent(\H)$, so
$\cent(\H)=\{\exp\,tZ\mid t\in\R\}$.
If $f:\H\ra L^1$ is a Lipschitz map, then for a full
measure set of points $p\in \H$
\begin{equation}
\label{eqn-weakdiff}
\liminf_{t\ra 0}\;\frac{d(f(p),f(p\exp tZ))}{d(p,p\exp tZ)} \,=\,0\,.
\end{equation}
\end{theorem}

\begin{proof}
Suppose  the theorem were false.   Then the set of points
$x\in \H$ such that 
$$
\liminf_{t\ra 0}\;\frac{d(f(x),f(x\exp tZ))}{d(x,x\exp tZ)} \,>\,0\,
$$
is measurable, and has positive measure. 
By countable additivity, it follows that there is a measurable set $Y$ of positive 
measure,
and constants $\la\in (0,1)$, $r\in (0,1)$, such that
if $p\in Y$ and $q=p\,\exp(tZ)$ for some $|t|<r$, then 
$d(f(p),f(q))\geq \la d(p,q)$.  

Let $p$
be a density  point of $Y$, where in addition the conclusion
of Theorem \ref{thm-pauls} holds.  We now blow up the pullback distance at $p$.
Take a sequence $\la_k\ra 0$,  define $\rho_k:\H\times\H\ra [0,\infty)$
by 
$$
\rho_k=  \frac{1}{\la_k}\,\left(  f\circ\ell_p\circ s_{\la_k} \right)^*d_{L^1}
\,,
$$
where $\ell_p:G\ra G$ is left translation by $p$, and $s_{\la_k}:G\ra G$
is the automorphism which scales by $\la_k$.
Since $f$ is Lipschitz, for some $C\in (0,\infty)$
we get $\rho_k\leq Cd$, and therefore by Arzela-Ascoli we may assume,
after passing to a subsequence if necessary,  that the sequence
$\{\rho_k\}$ converges uniformly on compact subsets of $G\times G$ to a
pseudo-distance $\rho_\infty$.   
Since $p$ was a density point of $Y$, it follows
that for any $x\in G$, $t\in \R$, 
\begin{equation}
\label{eqn-boundedbelow}
\rho_{\infty}(x,x\,\exp\,tZ)\geq \la\,d(x,x\,\exp\,tZ)\,.
\end{equation}

By Corollary \ref{cor-paulspluskakutani}, the pseudo-distance $\rho_\infty$
is induced by a map $f_\om:\H\ra L^1$, which by the choice of $p$, restricts
to a geodesic map on every line $L\in \L(\H)$.  If $\Si$ is the cut measure
associated with $f_{\om}$, then Proposition \ref{prop-geodesicmonotone}
implies that $\Si$-a.e. cut $E\in\cut(\H)$ is monotone; then 
Theorem \ref{thm-monotonehalfspace} gives that $\Si$ a.e. $E\in \cut(\H)$
is a half-space.

For almost every $z\in \R^2$, the restriction of the cut measure $\Si$ to the
fiber $\pi^{-1}(z)$ is well-defined, and supported on monotone cuts (since
the intersection of a fiber with a half-space is monotone).  Therefore
if $x_1,x_2,x_3\in \pi^{-1}(z)$ are in linear order, then 
$\rho_{\infty}(x_1,x_3)=\rho_{\infty}(x_1,x_2)+\rho_{\infty}(x_2,x_3)$.
Combining this with (\ref{eqn-boundedbelow}), for $n\in \N$
we get 
$$
\rho_{\infty}(x,x\,\exp nZ)=n\rho_{\infty}(x,x\,\exp Z)\geq n\la\,d(x,x\,\exp Z)\,.
$$
  This contradicts
the Lipschitz condition, since $d(x,x\exp nZ)\simeq \sqrt{n}$, see 
(\ref{eqn-verticaldistance}).

\end{proof}

\section{Uniqueness of cut measures}
\label{sec-injectivity}

Our main goal in this section is to show that under appropriate conditions, there
is a unique cut measure inducing a given cut metric.  Since the assignment $\Si\mapsto 
d_{\Si}$ is linear, it is natural to investigate injectivity in a linear framework,
and for this reason we will work with signed measures in this section.

\subsection*{The setup}
Throughout this section, $\Si$ will be a signed cut measure on $\H$
 supported on half-spaces. (Recall from 
Subsection \ref{subsec-cutmetrics} that the definitions of cut measure
and cut metric adapt directly to signed measures.)  We let
$d_{\Si}^{\horiz}$ denote the restriction of the cut metric $d_{\Si}$
to the set of horizontal pairs $\horiz(\H)\subset\H\times\H$.  
Let $\Si=\Si_+-\Si_-$
be the decomposition of $\Si$ into its positive and negative parts,
and let $|\Si|=\Si_++\Si_-$ be the absolute value of $\Si$.
 Since opposite half-spaces yield the same elementary cut metric
up to sets of measure zero, we may symmetrize $\Si$ so that it
is invariant under interchange of opposite components, without affecting the
associated cut metric $d_{\Si}$.  An alternative way to view this is to 
pushforward $\Si$ under the $2$-to-$1$ map $\hs(\H)\ra\P$ from the space of half-spaces
to planes, which sends each connected component of $\H\setminus P$ to $P$.
This pushforward operation induces a bijection between symmetrized cut measures
and measures on $\P$. 
We will often find it more convenient to work with measures on $\P$ rather than
symmetrized cut measures.

\bigskip
\begin{definition}
Let $\Si$ be a signed cut measure supported on half-spaces.
Then $\Si$ is
 {\bf Lipschitz} if there is a constant $C\in [0,\infty)$ such that
$d_{|\Si|}\leq C\,d$.  The {\bf Lipschitz constant of $\Si$}
is the infimum $\Lip(\Si)$ of such constants $C$.
\end{definition}
Note that the inequality is to be interpreted in accordance with the definition
of $d_{\Si}$, i.e. there should be a full measure subset $Z\subset\H$ such that
$d_{\Si}(x_1,x_2)\leq C\,d(x_1,x_2)$ for every $x_1,\,x_2\in Z$.

Our main objective in this section is:
\begin{theorem}
\label{thm-maininjectivity}
The linear map $\Si\mapsto d_{\Si}^{\horiz}$ is injective on symmetrized Lipschitz
signed cut measures supported on half-spaces.
\end{theorem}

 Let $\Si=\Si_v+\Si_h$ be the decomposition into
the parts supported on vertical
and horizontal half-spaces, respectively.   To simply terminology slightly,
we will use {\bf horizontal cut measure} (respectively {\bf vertical cut measure}
to refer to a measure supported on horizontal (resp. vertical) half-spaces.
We will first treat the injectivity question for the restricted operators
$\Si_v\mapsto d_{\Si_v}^{\horiz}$ and $\Si_h\mapsto d_{\Si_h}^
{\horiz}$, before demonstrating
injectivity in the general case.

\bigskip\bigskip
\subsection*{Estimates on horizontal cut measures}
We now fix a symmetrized horizontal cut measure $\Si$,
i.e. $\Si$ is supported on horizontal half-spaces.  
We will view $\Si$ as a Radon measure on the manifold $\P_h$
of horizontal planes, which we identify with $\H$  by the diffeomorphism
$\H\ra\P_h$ which sends $x\in\H$ to the unique horizontal plane centered at 
$x$.

From now until Theorem \ref{thm-horizontalinjectivity}   below, we will assume that $\Si$ is
absolutely continuous with respect to $\l$, so $\Si=u\,\l$ where
$u:\H\ra\R$ is a locally integrable function.

For $x\in\H$, let $P_x\subset\H$ denote the horizontal plane centered at
$x$.  If $x_1$ and $x_2$ are distinct  points lying on a line $L\in \L(\H)$, 
we define $P_{x_1,x_2}$ to be the union of the horizontal planes $P_x$, where
$x$ ranges over the interval $(x_1,x_2)\subset L$.  Since $x\in P_y$ if and 
only if $y\in P_x$, we have 
$P_{x_1,x_2}\;=\;\{y\in\H\mid P_y\cap (x_1,x_2)\neq\emptyset\}$.
The complement of the union 
 $P_{x_1}\cup P_{x_2}$ has  four wedge-shaped connected 
components, precisely two of which are ``horizontal'', in the sense that
they intersect each coset of the center in an interval.
 Neglecting the intersection with
$L$, the set $P_{x_1,x_2}$ coincides with the union of these two horizontal
components.

\begin{lemma}
\label{lem-distpx1x2} 
There is a full measure subset $Z\subset\H$
 such that
if  $(x_1,\,x_2)\in Z\times Z$ is a horizontal pair, then
$$
d_{\Si}(x_1,x_2)\,=\,\Si(P_{x_1,x_2})\,=\,\int_{P_{x_1,x_2}}\;u\,d\l\,\leq\,
\Lip(\Si)\,d(x_1,x_2)\,.
$$
\end{lemma}
\begin{proof}
From the definition of the cut metric, the distance $d_{\Si}(x_1,x_2)$
is given by
$$
\int_{\cut(\H)}\;|\chi_E(x_1)-\chi_E(x_2)|\,d\Si(E)\,,
$$
when $x_1$ and $x_2$ belong to a full measure subset of $\H$.   If 
$P$ is a plane disjoint from a horizontal pair $(x_1,x_2)$, and $E$ is a 
half-space component 
of $\H\setminus P$, then $|\chi_E(x_1)-\chi_E(x_2)|$
is nonzero precisely when $P\cap (x_1,x_2)\neq\emptyset$.   Since the
set of planes passing through $x_1$ or $x_2$ has $|\Si|$-measure zero, it follows that
$
d_{\Si}(x_1,x_2)\;=\;\Si(P_{x_1,x_2})\,.
$
\end{proof}

\bigskip
For $x\in\H$, let $s_x:\R^2\ra\H$ be the inverse of $\pi\restr_{P_x}$, and let
$d_{\pi(x)}:\R^2\ra[0,\infty)$ be the distance from $\pi(x)\in\R^2$. 

\bigskip
\begin{lemma}
\label{lem-lipest1}
Suppose in addition that $\Si$ is Lipschitz.
There is a constant $C$ depending only on $\H$, with the following property.
\begin{enumerate}
\item For almost every $x\in \H$, and for almost every line $L\in\L(\H)$ passing
through $x$, the composition $u\circ s_x:\R^2\ra\R$ is measurable.  Moreover,
if $d_{\pi(L)}:\R^2\ra[0,\infty)$ is the distance from $\pi(L)$, then  
$$
\int_{\R^2}\,|u\circ s_x|\,d_{\pi(L)}\,d\l\,\leq\,\Lip(\Si)\,.
$$
\item For almost every $x\in \H$,  the composition $u\circ s_x:\R^2\ra\R$ is measurable,  
$$
\int_{\R^2}\,|u\circ s_x|\,d_{\pi(x)}\,d\l\,\leq\,C\,\Lip(\Si)\,.
$$
\item If in addition  $d_{\Si}\,=\,0$, then for almost every $x\in\H$ and almost every
line $L$ passing through $x$, 
$$
\int_{\R^2}\,(u\circ s_x)\,d_{\pi(L)}\,d\l\,=\,\int_{\R^2}\,(u\circ s_x)\,d_{\pi(x)}\,d\l
\,=\,0\,.
$$
\end{enumerate}

\end{lemma}
\begin{proof}
Let $Z\subset\H$ be as in Lemma \ref{lem-distpx1x2}.
By Fubini's theorem, for almost every $L\in\L(\H)$, the intersection 
$L\cap Z$ has full (linear) measure in $L$.  Now let $\ga:\R\ra L$ 
be a unit speed parametrization, and let $S:\R^2\times\R\ra\H$
be defined by $S(y,t)=s_{\ga(t)}(y)$.  By Lemma \ref{lem-enclosedarea},
we have $s_{\ga(t)}(y)=s_x(y)\,\exp(A(y)Z)$, where $A(y)$ is the 
signed area enclosed by the triangle with vertices $\pi(x),\,\pi(\ga(t)),\,\pi(y)$;
this implies that
 the Jacobian of $S:(\R^2\times\R,\l)\ra (\H,\l)$ is $J(y,t)=d_{\pi(L)}(y)$.
Therefore by Lemma \ref{lem-distpx1x2} and the change of variables formula,
\begin{equation}
\label{eqn-changevariables}
d_{\Si}(\ga(t_1),\ga(t_2))=\int_{P_{\ga(t_1),\ga(t_2)}}\;u\,d\l\;
\end{equation}
$$
=\;
\int_{t_1}^{t_2}\int_{\R^2}\;(u\circ S)(y,t)\,d_{\pi(L)}(y)d\l(y)d\l(t)\,.
$$
Since $d_{|\Si|}(\ga(t_1),\ga(t_2))\leq \Lip(\Si)\,|t_1-t_2|$, 
by Fubini's theorem and the Fundamental Theorem of Calculus, 
it follows that for $\l$-a.e. $t\in\R$, we have
$$
\int_{\R^2}\,|u\circ s_{\ga(t)}(y,t)|\,d_{\pi(L)}(y)\;d\l(y)\leq \Lip(\Si)\,.
$$
Applying Fubini's theorem once again, we get (1).

Part (2) of the lemma follows from part (1) by integrating over the lines 
passing through $x$.

Part (3) follows from (\ref{eqn-changevariables}), Fubini's theorem, and the 
Fundamental Theorem of Calculus.

\end{proof}

\begin{remark}
Part (3) of Lemma \ref{lem-lipest1} generalizes to a statement  about
arbitrary Lipschitz
signed horizontal cut measures.   The metric differentiation result of 
Pauls generalizes to signed Lipschitz distance functions such as $d_{\Si}$.
It gives
rise to a measurable function on the collection $\{(L,x)\in  \L(\H)\times \H
\mid x\in L\}$ of pointed lines, which compares the infinitesimal behavior of $d_{\Si}$
along $L$, near $x$, with that of $d$.
 For almost every $x\in\H$ and almost every line
$L\in\L(\H)$ passing through $x$, the integral $\int_{\R^2}\;(u\circ s_x)\,d_{\pi(L)}\,
d\l$ agrees with this function. 

\end{remark}

\bigskip
\begin{lemma}
\label{lem-lipest2}
There is a constant $C$ which depends only on the geometry of $\H$,
such that if $\Si$ is Lipschitz, then:
\begin{enumerate}
\item  For every $L\in\L(\R^2)$, $r\in(0,\infty)$, the
strip $S=\{y\in\R^2\mid d(y,L)<r\}\subset\R^2$ satisfies
$$
|\Si|(\pi^{-1}(S))\;\leq\; \width(S)\Lip(\Si)\,=\,2r\Lip(\Si)\,.
$$
\item 
\label{item-integrable} 
If $\phi:\H\ra\R$ is a measurable function and 
$$
\sup_{x\in\H}\;|\phi(x)|(1+|\pi(x)|)^p\,<\,\infty
$$ for some 
$p>1$, then $\phi\,u\in L^1(\H,\l)$.  
\end{enumerate}
\end{lemma}
\begin{proof}
Suppose $x,\,y\in\H$ are  points lying on a line $L\in\L(\H)$ at distance
$d(x,y)=R>0$, and
$L_x,\,L_y\in\L(\H)$ are the lines intersecting
 $L$ orthogonally at $x$ and $y$, respectively.  Now let $y_1,\,y_2\in L_y$
be the two points  in $L_y$ at distance $r$ from $y$, and choose $x'\in L_x$.

Lemma \ref{lem-enclosedarea} implies that
the intersection of $P_{y_1,y_2}$
with the fiber $\pi^{-1}(\pi(x'))=x'\,\exp(\R Z)$ is of the form
$x'\exp (a,b)Z$, where $(a,b)$ is an interval of length $Rr$ shifted by 
the signed area of the triangle with vertices
$\pi(x),\pi(x'),\pi(y)$, i.e. $ \pm \frac12 d(x,x')R$.  Therefore if 
$r>|d(x,x')|$, if we hold 
$x$ and $L$ fixed while letting $R$ tend
to infinity, the set $P_{y_1,y_2}$
 will  contain more and more of the  fiber $\pi^{-1}(\pi(x'))=x'\exp\R Z$.  

Consider the strip $S\,=\,\{p\in\R^2\mid d(p,\pi(L))<r\}$.  If $K\subset\H$ is any compact
subset of $\pi^{-1}(S)$, then the discussion above implies
that when $R$ is sufficiently large, $P_{y_1,y_2}$ will contain $K$.  If we
choose a sequence of horizontal pairs $\{(y_1^k,y_2^k)\}$ which converge
to $(y_1,y_2)$, such that the conclusion of Lemma \ref{lem-distpx1x2} holds
for each of the pairs $(y_1^k,y_2^k)$, then we conclude that
$$
|\Si|(P_{y_1,y_2})\leq \liminf_{k\ra\infty}\,|\Si|(P_{y_1^k,y_2^k})
\,\leq\,\liminf_{k\ra\infty}\,\Lip(\Si)\,d(y_1^k,y_2^k)\,=\,\Lip(\Si)\,2r\,.
$$
In this case we get 
$$
|\Si|(K)\leq 
|\Si|(P_{y_1,y_2})\leq 2\,\Lip(\Si)\,r\,=\,\Lip(\Si)\width(S)\,.
$$
As $K$ was arbitrary (1) follows.

To prove (2), for $k\in\N$,  let $S_k$ be the double
strip 
$$
\{(x,y)\in \R^2\mid y\in (-k,-(k-1))\cup(k-1,k)\}\,.
$$
Then (1) gives 
$$
\int_{\pi^{-1}(S_k)}\,|\phi\, u|\,d\l\,\leq \,2\,(1+(k-1))^{-p}\,\Lip(\Si)\,,
$$
so
$$
\int_{\H}\;|\phi\,u|\,d\l\,=\,\sum_{k=1}^\infty\int_{\pi^{-1}(S_k)}\;
|\phi\,u|\;d\l\;\leq\; 2\Lip(\Si)\,\sum_{k=1}^\infty\,(1-(k-1))^{-p}\,<\,\infty\,.
$$

\end{proof}

\bigskip
\subsection*{Injectivity for horizontal cut measures}
Now suppose $\Si=u\,\l$ is a Lipschitz signed horizontal cut measure such that
$d_{\Si}^{\horiz}\,=\,0$, and
let $K$ be the distribution on $\H$ defined by the 
linear functional 
$$
\phi\mapsto\int_{\R^2}\,(\phi\circ s_e)\,d_{\pi(e)}\,d\l\,=\,
\int_{\R^2}\,(\phi\circ s_e)(x)\,|x|\,d\l(x)\,.
$$
By Lemma \ref{lem-lipest1},  the convolution $u*K$ is well-defined, and equals zero.
The convolution operator $\phi\mapsto \phi*K$ was studied 
in \cite{strichartz}.  Before proceeding, we 
briefly summarize the relevant conclusions from that paper.

To conform with the notation from \cite{strichartz}, we let $z=\pi:\H\ra\R^2\simeq\c$,
 and $t:\H\ra\R$ be the function given
by $p=s_e(\pi(p))\,\exp(tZ)=s_e(z(p))\,\exp(tZ)$.  Strichartz works with the
general Heisenberg group of dimension $2n+1$, so we are in the $n=1$ case.  
For $\la>0$, $k$ a nonnegative integer, and $\eps=\pm 1$, let 
\begin{equation}
\label{eqn-defphi}
\phi_{\la,k,\eps}(z,t)
\end{equation}
$$
\,=\,
(2\pi)^{n+1}\frac{\la^n}{(n+2k)^{n+1}}
\;\exp\left(-\frac{i\eps\la t}{n+2k}\right)\;
\exp\left(-\frac{\la|z|^2}{4(n+2k)}\right)\;
L^{n-1}_k\left(\frac{\la|z|^2}{2(n+2k)}\right)\,,
$$
where $L^{n-1}_k$ is a Laguerre polynomial.  The exponential decay
in $|z|$, together with Lemmas \ref{lem-lipest1} and
 \ref{lem-lipest2} are sufficient
to justify the calculations that arise below.

It was shown in \cite{strichartz} that  convolution operators with radial kernels (i.e.
kernels that are functions of $|z|$ and $t$) commute, so for instance we have
$$
*K*\phi_{\la,k,\eps}\,=\,*\phi_{\la,k,\eps}*K\,.
$$

The convolution operator with $K$ as above
was considered in Section 5, Example 2, with $n=1$, $\al=-1$ and $\be=1$. (Although
there it was assumed that
$0<\realpart \al < 2n$, the calculations are valid when $\al=-1$.)
It is shown there that for all $(\la,k,\eps)\in (0,\infty)\times \Z_{\geq 0}\times\{\pm 1\}$,
$$
\phi_{\la,k,\eps}*K\,=\,\kappa(\la,k,\eps)\,\phi_{\la,k,\eps}\,,
$$
 where
$\kappa(\la,k,\eps)$ is $m(\la(n+2k),\eps\la)$ in the notation of 
\cite[(2.35),(2.36), (5.9)]{strichartz}.  His calculations in (5.10), (5.13')
imply that $\kappa(\la,k,\eps)$ is nonzero for all $(\la,k,\eps)$. 
From this and the commutativity mentioned above, it follows that
$$
u*\phi_{\la,k,\eps}\,=\,\frac{1}{\kappa(\la,k,\eps)}\;u*\phi_{\la,k,\eps}*K
\,=\,\frac{1}{\kappa(\la,k,\eps)}(u*K)*\phi_{\la,k,\eps}\,=\,0\,,
$$
since $u*K=0$.  Therefore, using the exponential
decay in (\ref{eqn-defphi}), for any smooth
compactly supported function $v:(0,\infty)\ra (0,\infty)$, convolution
of $u$ with 
$$
V\;=\;
\int_{(0,\infty)}\,\sum_{k=1}^N\,\sum_{\eps=\pm 1}\;\phi_{\la,k,\eps}\,v(\la)\,d\la
$$
is also zero.   
It is not hard to see that such functions $V$ are dense among radial
functions in the Schwartz
space $\S(\H)$ of rapidly decreasing functions, and this implies that $u=0$.

\bigskip
\begin{theorem}
\label{thm-horizontalinjectivity}
If $\Si$ is a Lipschitz signed horizontal cut measure
such that
$d_{\Si}^{\horiz}\,=\,0$, then $\Si\,=\,0$.
\end{theorem}
The theorem follows from the above discussion when $\Si$ is absolutely
continuous with respect to $\l$.  The general case follows from this, 
by an approximation argument:

\begin{lemma}
\label{lem-convolution}
There is a sequence of smooth functions $\{\rho_k:\H\ra\R\}$
such that the sequence of measures $\{\Si_k=\rho_k\,\l\}$ converges
weakly to $\Si$, and each $\Si_k$ is a  Lipschitz
signed cut measure with vanishing cut metric.
\end{lemma}
\begin{proof}
Due to the $\H$-invariance of the setup, 
for any $g\in \H$, the pushforward of $\Si$  under  left translation
 $(\ell_g)_*\Si$ is also a Lipschitz horizontal cut measure with vanishing cut metric
on horizontal pairs.  Therefore
the same will be true of any convolution $\phi*\Si$, where $\phi$ is a compactly
supported continuous function.  Now let $\Si_k=\phi_k*\Si$, where $\{\phi_k\}$ 
is an appropriate sequence of smooth compactly supported functions converging 
weakly to a Dirac mass.
\end{proof}

\bigskip
\begin{corollary}
\label{cor-transinvzero}
If $\Si$ is a Lipschitz signed horizontal cut measure, and the restriction
of $d_{\Si}$ to horizontal pairs
is invariant under translation by the center, then $\Si=0$.
\end{corollary}
\begin{proof}
By the same smoothing argument as in Lemma \ref{lem-convolution}, it suffices
to treat the case when $\Si$ is absolutely continuous with respect to $\l$, 
so $\Si=u\,\l$.  

Pick a central element $g\in\exp \R Z$.   Define $\Si'$ to be the different
of signed measures $\Si-(\ell_g)_*\Si$.  By linearity $d_{\Si'}^{\horiz}=0$,
and so by Theorem \ref{thm-horizontalinjectivity} we have $\Si'=0$.  
Therefore $\Si$ is 
invariant under translation by the center $\exp\R Z$.  On the other hand,
by Lemma \ref{lem-lipest2}, the pushforward
  $\pi_*\Si$ is a Radon measure on $\R^2$.  This forces $\Si=0$.
\end{proof}

\bigskip
\bigskip
\subsection*{Injectivity for vertical cut measures}
We now suppose $\Si$ is a symmetrized  Lipschitz signed cut measure supported
on vertical half-spaces.

\begin{lemma}
\label{lem-verticalinjectivity}
If $d_{\Si}^{\horiz}=0$, then $\Si=0$.
\end{lemma}
\begin{remark}
This lemma is not needed in the proof of the main theorem,
so some readers may prefer to skip it.
It does, however, give some additional information about the situation of the theorem.
\end{remark}

\begin{proof}

As with horizontal cut measures, we prefer to work with a measure on the
space of vertical planes $\P_v$, rather than a symmetric cut measure.
Moreover, since vertical planes are in obvious bijection with lines in 
$\R^2$, we may reformulate this as a question about a
signed cut measure $\si$ on the space of lines $\L(\R^2)$, and the associated
cut metric $d_{\si}$ on $\R^2$. 
By a smoothing argument as in Lemma \ref{lem-convolution}, we may 
assume that $\si$ is of the form $u\,\l$, where $u:\L(\R^2)\ra\R$ is a smooth
function, and $\l$ is Haar  measure on $\L(\R^2)$ (when viewed as a homogeneous
space of the isometry group $\isom(\R^2)$).   Note that the asymptotic
behavior of the cut metric $d_{\si}$ near a point $x\in \R^2$ agrees with the one induced
by a translation invariant cut measure $\si_x$ on $\R^2$, where $\si_x$ 
is determined by the density function $u$ restricted to the set of lines
passing through $x$; hence we are reduced to the case when $\si$
is translation invariant.    

A translation invariant cut measure $\si$ is obtained as follows.  For each direction
$v$ in $\R^2$,  
there is a unique translation invariant  measure $\tau_v$ on $\L(\R^2)$ 
supported on the set of lines parallel to $v$, normalized
such that the $d_{\tau_v}$-distance between
two lines $L_1,\,L_2$ parallel to $v$ is the same as their Euclidean distance.  
A general translation invariant cut measure $\si$ is obtained as a superposition
of the $\tau_v$'s, or as the pushforward of a signed Radon
measure $\mu$ on $\R P^1$ under the map $v\mapsto \tau_v$.  A calculation shows that
the induced distance on $\R^2$ is homogeneous of degree $1$, and that
if $\xi$ is a unit vector, then 
$$
d_{\si}(0,\xi)\,=\,\int_{\R P^1}\,|\sin\angle (\xi,v)|\,d\mu(v)\,.
$$  
Thus the operator $\si\mapsto d_{\si}$ is equivalent to the convolution
operator $\mu\mapsto \mu *|\sin\th|$ on $\R P^1$.  Direct calculation shows
that this is injective for complex exponentials $\exp(ik\th):\R P^1\ra\R$
for $k$ even, and this implies injectivity in general.  
\end{proof}

\bigskip
\bigskip
\subsection*{Proof of Theorem \ref{thm-maininjectivity}}

Suppose $\Si=\Si_v+\Si_h$ is a symmetrized Lipschitz signed  cut measure
and $d_{\Si}^{\horiz}=0$.  Since  
both $d_{\Si}^{\horiz}$ and $d_{\Si_v}^{\horiz}$ are invariant
under translation by the center, the same is true of
 $d_{\Si_h}^{\horiz}=d_{\Si}^{\horiz}-d_{\Si_v}^{\horiz}$.  By
Corollary \ref{cor-transinvzero}, we have $\Si_h=0$.  Then $d_{\Si_v}=0$,
and Lemma \ref{lem-verticalinjectivity} implies that $\Si_v=0$ as well.

\bigskip
\bigskip
\section{Cut measures which are standard on lines}
\label{sec-standardonlines}
When $\Si$ is the cut measure arising from a map $\H\ra L^1$ which
comes from Pauls' metric diffentiation theorem, then we know that
the   restriction of $d_{\Si}$ to lines is a constant
multiple of the Heisenberg metric.  Using the injectivity statement
in Theorem \ref{thm-maininjectivity}, we get:

\begin{theorem}
\label{thm-standardcutmeasure}
Suppose $\Si$ is a cut measure on $\H$ such that:
\begin{itemize}
\setlength{\itemsep}{.75ex}%
\setlength{\parskip}{.75ex}%
\item  The cut metric $d_{\Si}$ is bounded by a multiple of $d$:
$d_{\Si}\leq C\,d$.
\item The restriction of $d_{\Si}$ to almost every  line $L\in \L(\H)$ is
a constant multiple of $d$: 
$$
d_{\Si}\restr_{L}\;=\;c_L\,d\restr_L\,.
$$
\end{itemize}
Then $\Si$ is supported on vertical half-spaces, and moreover, its 
symmetrization is translation invariant.  
\end{theorem}
\begin{proof}
Applying Proposition \ref{prop-geodesicmonotone}, it follows that 
$\Si$ is supported on monotone cuts, and hence by 
Theorem \ref{thm-monotonehalfspace} on vertical and horizontal half-spaces.
We may assume that $\Si$ is symmetric.

The Lipschitz condition implies that if $L_1$ and $L_2$ are parallel
lines, then the multiples $c_{L_1}$ and $c_{L_2}$ coincide, because the
lines diverge sublinearly.  Thus 
 $d_{\Si}^{\horiz}$  is invariant under translation.  Then $d_{\Si_h}^{\horiz}
=d_{\Si}^{\horiz}-d_{\Si_v}^{\horiz}$
is invariant under vertical translation,  and Corollary \ref{cor-transinvzero}
gives $\Si_h=0$.  Therefore $d_{\Si_v}^{\horiz}$ is $\H$-invariant, so
 Lemma \ref{lem-verticalinjectivity} implies that $\Si=\Si_v$ is
$\H$-invariant. 
\end{proof}

\bigskip
{\em Proof of Theorem \ref{thm-maindiff}.}
Let $f:\H\ra L^1$ be a Lipschitz map, and let 
$\rho_{x,\la}$ be as in the statement of Theorem \ref{thm-maindiff}. 

By Theorem \ref{thm-pauls}, for almost every point $x\in \H$, there
is a norm $\|\cdot\|_x$ on $\R^2$, such that $\rho_{x,\la}\restr_{\horiz(\H)}$
converges uniformly on compact sets to the pseudo-metric 
$(z_1,z_2)\mapsto \|\pi(z_1)-\pi(z_2)\|_x$, for all $(z_1,z_2)\in\horiz(\H)$.   
We claim that the same statement
holds for arbitrary pairs.  If not, since the family $\{\rho_{x,\la}\}_{\la\in (0,\infty)}$
is uniformly Lipschitz, by the Arzela-Ascoli theorem we may find a 
sequence $\{\la_k\}\ra 0$ such that $\rho_{x,\la_k}$ converges uniformly on 
compact subsets of $\H\times\H$ to a limiting pseudo-distance $\rho_\infty$,
such that 
\begin{equation}
\label{eqn-badone}
\rho_\infty(z_1,z_2)\neq \|\pi(z_1)-\pi(z_2)\|_x
\end{equation}
for some
$(z_1,z_2)\in\H\times\H$, while $\rho_\infty(z_1,z_2)\neq \|\pi(z_1)-\pi(z_2)\|_x$
for all $(z_1,z_2)\in\horiz(\H)$.  

By Corollary \ref{cor-paulspluskakutani}, the pseudo-distance $\rho_\infty$
is induced by a map $f_\om:\H\ra L^1$.  If $\Si$ is the cut measure
associated with $f_{\om}$, then $\Si$ satisfies the hypotheses
of  Theorem \ref{thm-standardcutmeasure}.  Therefore 
$\Si$ is supported on vertical half-spaces, which means that $d_{\Si}$
is the pullback of a metric from $\R^2$ by the projection map $\pi$.  This
contradicts (\ref{eqn-badone}).

\bigskip\bigskip
\bibliography{h2}

\def\cprime{$'$} \def\cprime{$'$} \def\cprime{$'$}
\begin{thebibliography}{CJLV08}

\bibitem[Amb01]{luigi1}
L.~Ambrosio.
\newblock Some fine properties of sets of finite perimeter in {A}hlfors regular
  metric measure spaces.
\newblock {\em Adv. Math.}, 159(1):51--67, 2001.

\bibitem[Amb02]{luigi2}
L.~Ambrosio.
\newblock Fine properties of sets of finite perimeter in doubling metric
  measure spaces.
\newblock {\em Set-Valued Anal.}, 10(2-3):111--128, 2002.
\newblock Calculus of variations, nonsmooth analysis and related topics.

\bibitem[AR98]{aumannrabani}
Yonatan Aumann and Yuval Rabani.
\newblock An {$O(\log k)$} approximate min-cut max-flow theorem and
  approximation algorithm.
\newblock {\em SIAM J. Comput.}, 27(1):291--301 (electronic), 1998.

\bibitem[Ass80]{assouad}
P.~Assouad.
\newblock Plongements isom\'etriques dans {$L\sp{1}$}: aspect analytique.
\newblock In {\em Initiation Seminar on Analysis: G. Choquet-M. Rogalski-J.
  Saint-Raymond, 19th Year: 1979/1980}, volume~41 of {\em Publ. Math. Univ.
  Pierre et Marie Curie}, pages Exp. No. 14, 23. Univ. Paris VI, Paris, 1980.

\bibitem[BL00]{benlin}
Y.~Benyamini and J.~Lindenstrauss.
\newblock {\em Geometric nonlinear functional analysis. {V}ol. 1}, volume~48 of
  {\em American Mathematical Society Colloquium Publications}.
\newblock American Mathematical Society, Providence, RI, 2000.

\bibitem[Bou85]{bourgain}
J.~Bourgain.
\newblock On {L}ipschitz embedding of finite metric spaces in {H}ilbert space.
\newblock {\em Israel J. Math.}, 52(1-2):46--52, 1985.

\bibitem[Che99]{Cheeger}
J.~Cheeger.
\newblock Differentiability of {L}ipschitz functions on metric measure spaces.
\newblock {\em Geom. Funct. Anal.}, 9(3):428--517, 1999.

\bibitem[CJLV08]{chakrabarti}
A.~Chakrabarti, A.~Jaffe, J.~R. Lee, and J.~Vincent.
\newblock Embeddings of topological graphs: Lossy invariants, linearization,
  and $2$-sums.
\newblock In {\em 49th Annual Symposium on Foundations of Computer Science},
  pages 761--770, 2008.

\bibitem[CK]{ckmetdiff}
J.~Cheeger and B.~Kleiner.
\newblock Metric differentiation for {P}{I} spaces.
\newblock In preparation.

\bibitem[CK06]{ckbv}
J.~Cheeger and B.~Kleiner.
\newblock Differentiating maps to ${L}^1$ and the geometry of {B}{V} functions.
\newblock math.MG/0611954, to appear in {\em Annals of Mathematics}, 2006.

\bibitem[CK08a]{ckdppi}
J.~Cheeger and B.~Kleiner.
\newblock Differentiation of {L}ipschitz maps from metric measure spaces to
  {B}anach spaces with the {R}adon {N}ikodym {P}roperty.
\newblock preprint, 2008.

\bibitem[CK08b]{laaksoembed}
J.~Cheeger and B.~Kleiner.
\newblock Embedding {L}aakso spaces in ${L}^1$.
\newblock In preparation, 2008.

\bibitem[CKN]{ckn}
J.~Cheeger, B.~Kleiner, and A.~Naor.
\newblock Compression bounds for {L}ipschitz maps from the {H}eisenberg group
  to ${L}^1$.
\newblock In preparation.

\bibitem[DCK72]{dacuna}
D.~Dacunha-Castelle and J.~L. Krivine.
\newblock Applications des ultraproduits \`a l'\'etude des espaces et des
  alg\`ebres de {B}anach.
\newblock {\em Studia Math.}, 41:315--334, 1972.

\bibitem[EFW06]{eskinfisherwhyte}
A.~Eskin, D.~Fisher, and K.~Whyte.
\newblock Coarse differentiation of quasiisometries i: spaces not
  quasiisometric to {c}ayley graphs.
\newblock Preprint, 2006.

\bibitem[FSSC01]{italians1}
B.~Franchi, R.~Serapioni, and F.~Serra~Cassano.
\newblock Rectifiability and perimeter in the {H}eisenberg group.
\newblock {\em Math. Ann.}, 321(3):479--531, 2001.

\bibitem[FSSC03]{italians2}
B.~Franchi, R.~Serapioni, and F.~Serra~Cassano.
\newblock On the structure of finite perimeter sets in step 2 {C}arnot groups.
\newblock {\em J. Geom. Anal.}, 13(3):421--466, 2003.

\bibitem[Gro93]{asyinv}
M.~Gromov.
\newblock Asymptotic invariants of infinite groups.
\newblock In {\em Geometric group theory, {V}ol.\ 2 ({S}ussex, 1991)}, volume
  182 of {\em London Math. Soc. Lecture Note Ser.}, pages 1--295. Cambridge
  Univ. Press, Cambridge, 1993.

\bibitem[Hei80]{heinrich}
S.~Heinrich.
\newblock Ultraproducts in {B}anach space theory.
\newblock {\em J. Reine Angew. Math.}, 313:72--104, 1980.

\bibitem[HK96]{HeKo}
J.~Heinonen and P.~Koskela.
\newblock From local to global in quasiconformal structures.
\newblock {\em Proc. Nat. Acad. Sci. USA}, 93:554--556, 1996.

\bibitem[HM82]{heinrichmankiewicz}
S.~Heinrich and P.~Mankiewicz.
\newblock Applications of ultrapowers to the uniform and {L}ipschitz
  classification of {B}anach spaces.
\newblock {\em Studia Math.}, 73(3):225--251, 1982.

\bibitem[Kak39]{kakutani}
S.~Kakutani.
\newblock Mean ergodic theorem in abstract {$(L)$}-spaces.
\newblock {\em Proc. Imp. Acad., Tokyo}, 15:121--123, 1939.

\bibitem[Kir94]{kirchheim}
B.~Kirchheim.
\newblock Rectifiable metric spaces: local structure and regularity of the
  {H}ausdorff measure.
\newblock {\em Proc. Amer. Math. Soc.}, 121(1):113--123, 1994.

\bibitem[KL97]{kleinerleeb}
B.~Kleiner and B.~Leeb.
\newblock Rigidity of quasi-isometries for symmetric spaces and {E}uclidean
  buildings.
\newblock {\em Inst. Hautes \'Etudes Sci. Publ. Math.}, (86):115--197 (1998),
  1997.

\bibitem[Lin02]{linialicm}
N.~Linial.
\newblock Finite metric-spaces---combinatorics, geometry and algorithms.
\newblock In {\em Proceedings of the International Congress of Mathematicians,
  Vol. III (Beijing, 2002)}, pages 573--586, Beijing, 2002. Higher Ed. Press.

\bibitem[LLR95]{liniallondonrabinovich}
Nathan Linial, Eran London, and Yuri Rabinovich.
\newblock The geometry of graphs and some of its algorithmic applications.
\newblock {\em Combinatorica}, 15(2):215--245, 1995.

\bibitem[LN06]{leenaor}
J.~Lee and A.~Naor.
\newblock ${L}^p$ metrics on the {H}eisenberg group and the {G}oemans-{L}inial
  conjecture.
\newblock Preprint, 2006.

\bibitem[LR07]{leeraghavendra}
J.~R. Lee and P.~Raghavendra.
\newblock Coarse differentiation and planar multi-flows.
\newblock {\em APPROX}, 2007.
\newblock To appear, Discrete and Computational Geometry.

\bibitem[Pan89]{pansu}
P.~Pansu.
\newblock M\'etriques de {C}arnot-{C}arath\'eodory et quasiisom\'etries des
  espaces sym\'etriques de rang un.
\newblock {\em Ann. of Math. (2)}, 129(1):1--60, 1989.

\bibitem[Pau01]{pauls}
S.~Pauls.
\newblock The large scale geometry of nilpotent {L}ie groups.
\newblock {\em Comm. Anal. Geom.}, 9(5):951--982, 2001.

\bibitem[Str91]{strichartz}
R.~S. Strichartz.
\newblock {$L\sp p$} harmonic analysis and {R}adon transforms on the
  {H}eisenberg group.
\newblock {\em J. Funct. Anal.}, 96(2):350--406, 1991.

\end{thebibliography}
\bibliographystyle{alpha}

\end{document}